\documentclass[12pt]{amsart}

\usepackage{Commands}
\usepackage{scrpage2}
\usepackage{jabbrv}
\allowdisplaybreaks

\numberwithin{equation}{section}
\title{Non-Autonomous Conformal Graph Directed Markov Systems}
\author{Jason Atnip}
\address{Department of Mathematics, University of North Texas, Denton, TX 76203-1430, USA}
\email{jason.atnip@unt.edu  \hspace*{0.42cm} \it Web: \rm http://atnipmath.com}

\date{}
\begin{document}
\begin{abstract}
	In this paper we introduce and develop the theory of non-autonomous graph directed Markov systems which is a generalization of the theory of conformal graph directed Markov systems of Mauldin and Urba\'nski, first presented in their book \cite{mauldin_graph_2003}, and the theory of non-autonomous conformal iterated function systems set forth by Rempe-Gillen and Urba\'nski in \cite{rempe-gillen_non-autonomous_2016}. We exhibit several large classes of functions for which Bowen's formula for Hausdorff dimension holds. In particular we consider weakly balanced finite systems, where we have some control over the growth of the derivatives, and ascending systems. Our results, particularly for ascending systems, generalize and go well beyond what is currently known for autonomous graph directed Markov systems and non-autonomous iterated function systems. We also provide an application to non-autonomous conformal dynamics by estimating the Hausdorff dimension of the Julia set of non-autonomous affine perturbations of an elliptic function from below.
\end{abstract}

\maketitle

\tableofcontents

\section{Introduction}
	The theory of iterated function systems and their many generalizations has a rich history dating back as far as the late 1940's to Moran  
	\cite{moran_additive_1946}, while their contemporary construction	begins with Hutchinson's 1981 paper, \cite{hutchinson_fractals_1981}. In the classical theory of iterated function systems one considers a finite collection of uniformly contracting similarity maps on a compact subset of Euclidean space. This theory was then expanded to consider infinite collections of maps which are allowed to be conformal mappings as opposed to only considering similarities. A map $\phi$ defined on an open connected subset of Euclidean space is said to be conformal if its derivative is a similarity map at each point. See \cite{mauldin_dimensions_1996} for a detailed development of this theory, which paved the way for conformal graph directed systems. 
	Graph directed systems were first pioneered by Mauldin and Williams in \cite{mauldin_hausdorff_1988} and then by Edgar and Mauldin in \cite{edgar_multifractal_1992}. The book of Mauldin and Urba\'nski on conformal graph directed Markov systems, \cite{mauldin_graph_2003}, represents the full generality of such autonomous systems. Recently, work has been done to extend the theory of autonomous conformal graph directed Markov systems to Carnot groups in \cite{chousionis_conformal_2016}. Conformal graph directed Markov systems have also been considered in the random setting by Roy and Urba\'nski in \cite{roy_random_2011}.
	
	The main focus of this article is to generalize the notion of non-autonomous conformal iterated function systems (NCIFS) first developed in \cite{rempe-gillen_non-autonomous_2016} by Rempe-Gillen and Urba\'nski to the setting of non-autonomous conformal graph directed Markov systems. In an autonomous system we consider a fixed collection of contraction mappings, and then look at all admissible compositions to define a limit set. In a non-autonomous system, we again look at the limit set generated by admissible compositions of contraction mappings, however we no longer consider a fixed set of mappings. The collections of mappings are in fact allowed to change at each time step. In our article we go beyond the generality of non-autonomous systems as presented in \cite{rempe-gillen_non-autonomous_2016} and allow for mappings whose domain and codomain are not the same space and allowed to change with each time step. In particular, our construction allows for a countable collection of compact connected subsets of Euclidean space rather than the finite collection of such spaces considered with autonomous graph directed Markov systems. We would also like to mention that, as a consequence of this increased generality, in this article we require many sequences of constants that were previously fixed values in lesser generality to be subexponential in growth, e.g. $\lim_{n\to\infty}1/n\log \#V_n=0$ where $V_n$ is the collection of vertices at time $n$. Some such constants are required in \cite{rempe-gillen_non-autonomous_2016}, but we will require many more here.     

	In this article we seek to prove that the Hausdorff dimension of the limit set of a system $\Phi$ is equal to the unique ``zero", $B_\Phi$, of the pressure function. Such a formula is often referred to as Bowen's formula, being first discovered by Rufus Bowen for the context of quasi-Fuchsian groups in 1979 in \cite{bowen_hausdorff_1979}. Since its conception Bowen's formula,
\begin{align*}
	\HD(J_\Phi)=B_\Phi,
\end{align*}
has been shown for all finite and infinite autonomous, as well as random, conformal iterated functions systems, and their generalizations graph directed Markov systems. See, for example, \cite{mauldin_graph_2003} for the autonomous setting and \cite{roy_random_2011} for the random setting. However, this is not the case for non-autonomous systems. In fact, there are examples of finite non-autonomous conformal iterated function systems for which Bowen's formula does not hold. For a general finite system we will need some sort of control over the growth rate of the alphabets at each time step. Specifically, Rempe-Gillen and Urba\'nski showed that, for finite systems, if $\lim_{n\to\infty}1/n \log\#I^{(n)}=0$, i.e. the system is subexponentially bounded, then Bowen's formula holds. For non-autonomous graph directed systems we obtain a similar result, however, we will need an additional balancing condition on the growth of the derivatives and a condition on the incidence matrices that allows for sufficiently many admissible words, whose definitions will be given in Section \ref{Sec: Subexp}. 
The main result of our paper is the following. 
\begin{theorem}\label{MainThm}
	If $\Phi$ is a finitely primitive non-autonomous conformal graph directed Markov system of subexponential growth
	\begin{align*}
		\limty{n}\frac{1}{n}\log\#I^{(n)}=0,
	\end{align*} 
	 such that
	 \begin{align*}
		 \limty{n}\frac{1}{n}\log\sup_{a,b\in I^{(n)}}\frac{\norm{D\phi_a^{(n)}}}{\norm{D\phi_b^{(n)}}}=0,
	 \end{align*}
	 then Bowen's formula holds, i.e. $\HD(J_\Phi)=B_\Phi$.
\end{theorem}
In general we are unable to remove the subexponential growth restriction on the alphabets without imposing other restrictive assumptions. In fact, for any $\ep>0$ and $0<s<t<d$, Rempe-Gillen and Urba\'nski construct a finite non-autonomous conformal iterated function system $\Phi$ such that 
\begin{align*}
	\limsupty{n}\frac{1}{n}\log\#I^{(n)}\leq \ep,
\end{align*} 
$\HD(J_\Phi)=s$, and $B_\Phi=t$.
The remainder of the article is devoted to removing any or all of these conditions on the growth rate of the alphabet or size of the derivatives. In particular we introduce a class of ascending systems for which $I^{(n)}\sub I^{(n+1)}$ for each $n\in\NN$ so that at each time step we consider more and more maps without losing any information from the previous stages. We then show that Bowen's formula holds for all finite and infinite systems without requiring any restrictions on the derivatives or alphabets. Specifically we show the following. 
\begin{theorem}
	If $\Phi$ is a finitely primitive, ascending non-autonomous conformal graph directed Markov system, then Bowen's formula holds. 
\end{theorem}
In the final section we present an application of our theory to non-autonomous conformal dynamics. We show that given a particular elliptic function then for affine perturbations sufficiently close to the original function then the non-autonomous Julia set has dimension greater than or equal to some positive number depending on the original function. Specifically we show the following.
\begin{theorem}
	Let $f_0$ be an elliptic function. Then there are $\ep,\dl>0$ such that if $\lm_n, \lm_n^{-1}\in B(1,\dl)$ and $c_n\in B(0,\ep)$ for all $n\in\NN$ then 
	\begin{align*}
	\HD(\jl(F_{\lm,c}))\geq \frac{2q}{q+1}
	\end{align*}
	where $q$ is the maximum multiplicity of the poles of $f_0$ and $F_{\lm,c}$ is defined to have iterates 
	\begin{align*}
		F^n_{\lm,c}=f_n\circ\dots\circ f_1, \text{ with } f_n=\lm_n\cdot f_0+c_n
	\end{align*}
	for each $n\in\NN$.
	
\end{theorem}

\subsection*{Structure of the Paper}
In Section \ref{Sec: Setup} we establish the fundamental definitions and necessary assumptions for the objects which we will study, e.g. non-autonomous graph directed systems, alphabets, and limit sets. In Section \ref{Sec: Subexp} we establish the definitions of subexponentially bounded and finitely primitive systems and present results detailing the interplay of these two assumptions. The pressure functions and their behavior are then described in Section \ref{Sec: Pressure} as well as a result showing that the Bowen dimension of a system is always greater than or equal to the Hausdorff dimension of the limit set. Section \ref{Sec: Gen Low Bdd} is devoted entirely to the proof of a technical theorem giving a general lower bound for the dimension of the limit set which is then used in Section \ref{sec: balancing} to establish Bowen's formula for sufficiently well balanced finite systems. In Section \ref{section Rem-Urb NCIFS} we extend Rempe-Gillen and Urba\'nski's results for infinite and subexponentially bounded stationary non-autonomous iterated function systems, without balancing conditions, to the case of non-stationary non-autonomous iterated function systems. In Section \ref{Sec: Haus meas} we consider the Hausdorff measure of the limit sets of a class of uniformly finite systems. Sections \ref{Sec: Relax Subexp} and \ref{Ascending} are devoted to relaxing the subexponential boundedness and balancing conditions by providing similar conditions and introducing ascending systems. The final section, Section \ref{Sec: Examples}, provides examples of non-autonomous conformal graph directed Markov systems as well as an application to conformal dynamics.  

In this paper we let $``\const"$ denote an arbitrary positive constant, which may change from line to line. 

\section{Setup and Main Assumptions}\label{Sec: Setup}
We begin by introducing several definitions and notations necessary to give a proper description of non-autonomous graph directed systems. As in the case of random and autonomous graph directed Markov systems, a non-autonomous graph directed Markov system, is based on a directed multigraph $(V,E,i,t)$ and a sequence of edge incidence matrices $\seq{A^{(n)}}$. $V$ will be a countable set of vertices and $E$ will be a countable set of edges. Each of the sets $V$ and $E$ may, and most likely will be, infinite. The functions $i,t:E\to V$ give the direction on the graph such that for any given edge $e\in E$, $i(e)$ denotes the initial vertex of $e$ and $t(e)$ denotes the terminal vertex of $e$.

Furthermore, we assume that there are sequences $\seq{E_n}\sub E$ and $\left(V_n\right)_{n\geq0}\sub V$, with
\begin{align*}
	\#V_n<\infty\quad \text{ for each }n\in\NN,
\end{align*}
such that for each $n\geq 1$ and each edge $e\in E_n$, we have
\begin{align*}
	i(e)\in V_{n-1} \spand t(e)\in V_{n}.
\end{align*} 
Here $E_n$ may be infinite. Furthermore, without loss of generality, we assume that 
\begin{align*}
	V=\union_{n\geq 0}V_n.
\end{align*}
Throughout the article, we fix $d\in\NN$, and for each $n\geq 0$ and $v\in V_n$, we let
\begin{align*}
\cX^{(n)}=\set{X_v^{(n)}:v\in V_n}
\end{align*}
where $X_v^{(n)}$ is a non-empty compact, connected subset of $\RR^d$ which is regularly closed, i.e.
\begin{align*}
	X_v^{(n)}=\ol{\intr{X_v^{(n)}}}.
\end{align*}
For fixed $n\in\NN$ we assume that the $X_v^{(n)}$ are disjoint ranging over $V_n$. However, this assumption is nonessential and is taken only for the ease of exposition and simplification of proofs. Remark 5.18 of \cite{chousionis_conformal_2016} describes a process for lifting an autonomous graph directed Markov system $\Phi$, in which the spaces $X_v$ are not necessarily disjoint, to a system $\Phi'$ such that $\Phi$ and $\Phi'$ have essentially the same limit set but the corresponding compact spaces $X_v'$ of $\Phi'$ are disjoint. This procedure is easily modified to fit our non-autonomous framework. 

For each time $n\geq 1$, we consider $E_n$ to be the alphabet at time $n$ and let 
\begin{align*}
	\Phi^{(n)}=\set{\phi_e^{(n)}:X_{t(e)}^{(n)}\to X_{i(e)}^{(n-1)}}_{e\in E_n}
\end{align*}
be a collection of functions whose domains and codomains depend upon the direction and time step of the multigraph. Later, for the sake of notational continuity, we will prefer the notation $I^{(n)}:=E_n$ for the alphabet at time $n$. Now for each $n\in\NN$ there is an edge incidence matrix 
\begin{align*}
	A^{(n)}:E_n\times E_{n+1}\to \set{0,1}
\end{align*} 
defined by $A^{(n)}_{ab}=1$ implies $t(a)=i(b)$. Notice that when  $t(a)=i(b)$, then the domain of $\phi_a^{(n)}$, $X_{t(a)}^{(n)}$, is equal to the codomain of $\phi_b^{(n+1)}$, $X_{i(b)}^{(n)}$, which shows that the composition $\phi_a^{(n)}\circ\phi_b^{(n+1)}$ is possible, though not necessarily allowable. If  $A^{(n)}_{ab}=1$ will we say that the composition $\phi_a^{(n)}\circ\phi_b^{(n+1)}$ is admissible, or allowable. In particular, this means that the letter $b$ is allowed to follow the letter $a$.

We now describe the many collections of words with which will be working. For $1\leq m\leq n\leq \infty$ let 
\begin{align*}
	E^{m,n}=\ds\prod_{j=m}^n E_j.
\end{align*}
If $m=1$, we simply write $E^n$. For any finite word $\om\in E^{m,n}$ for $1\leq m\leq n<\infty$, we let $\absval{\om}$ denote the length of $\om$, and we extend the notions of the functions $i,t$ by letting $i(\om):=i(\om_m)$ and $t(\om):=t(\om_n)$.
\begin{definition}
	For $1\leq m\leq n\leq \infty$, a word $\om\in E^{m,n}$ is called \textit{admissible} if $A^{(j)}_{\om_j\om_{j+1}}=1$ for all $m< j\leq n$, or equivalently if we have 
	\begin{align*}
		A^{m,j}_\om:=\ds\prod_{k=m}^{j}A^{(k)}_{\om_k\om_{k+1}}=1.
	\end{align*}
\end{definition}
We let $I^n_A$ denote the set of admissible all words of length $n$, for $1\leq n\leq \infty$, that is,
\begin{align*}
	I^n_A:=\set{\om\in E^\infty : A^{(j)}_{\om_j\om_{j+1}}=1 \text{ for all } 1\leq j\leq n}.
\end{align*} 
If there is no confusion about the sequence $A=\set{A^{(n)}}_{n\in\NN}$ of incidence matrices, for notational convenience we shall write $I^n$ instead. Now for $1<m\leq n\leq \infty$ we let $I^{m,n}$ denote the set of admissible words from time $m$ to time $n$ that are part of some infinite admissible word. To say this another way in terms of extensions, we write
\begin{align*}
	I^{m,n}:=\{\om\in E^{m,n} : A^{(j)}_{\om_j\om_{j+1}}&=1 \text{ for all }m\leq j\leq n-1,\\ 
	&\text{ and } \exists\al\in I^{m-1},\gm\in I_{n+1}^\infty \st \al\om\gm\in I^\infty\}.
\end{align*}
If $m=1$, we again abbreviate $I^n:=I^{1,n}$. We let $I^*$ denote the set of all finite admissible words which originate from $I^{(1)}$, that is
\begin{align*}
	I^*=\set{\om:\exists n\in\NN \st \om\in I^n}.
\end{align*}
For each finite word $\om=\om_m\om_{m+1}\dots\om_n\in I^{m,n}$ we associate the conformal map $\phi^{m,n}_\om$ given by 
\begin{align*}
	\phi^{m,n}_\om:=\phi_{\om_m}^{(m)}\circ\cdots\circ\phi_{\om_n}^{(n)}:X_{t(\om)}^{(n)}\to X_{i(\om)}^{(m)}; \qquad \phi_{\om_j}^{(j)}\in\Phi^{(j)}, m\leq j\leq n.
\end{align*}
From this moment on, for the sake of continuity of our notation and that of keeping with the notation in \cite{rempe-gillen_non-autonomous_2016}, we will let
\begin{align*}
	I^{(n)}:=E_n
\end{align*} 
denote the alphabet at time $n$. If no confusion arises, for the sake of notational convenience, we will simply denote $\phi_\om:=\phi_\om^{m,n}$ for $\om\in I^{m,n}$ and $\phi_a:=\phi_a^{(n)}$ for $a\in I^{(n)}$. In particular, we will use parentheses whenever we wish to talk about a single time step and no parentheses when we wish to discuss longer words. 

If we are interested in words originating from a particular vertex $v$, then for each of the alphabets defined above, we can formulate the analogous definition for each $1\leq m\leq n\leq \infty$, by letting
\begin{align*}
	I_v^{m,n}=\set{\om\in I^{m,n}:i(\om)=v}.
\end{align*}
Using this notation we clearly have that
\begin{align*}
	I^{(n)}=\union_{v\in V} I^{(n)}_v \quad \text{ and } \quad I^n=\union_{v\in V} I^n_v.
\end{align*}
If a map $\phi:X\to Y$ is conformal, or a similarity, we let $D\phi$ denote the usual derivative $\phi'$, or respectively the scaling factor, of $\phi$, and set 
\begin{align*}
	\norm{D\phi}=\sup\set{\absval{D\phi(x)}:x\inX}.
\end{align*} 
We are now ready to define our primary object of study.
\begin{definition}
	A \textit{non-autonomous conformal graph directed Markov system} (NCGDMS) $\Phi$ is a sequence of maps, incidence matrices, and spaces together with a multigraph denoted by 
	\begin{align*}
	\Phi=\left(\seq{\Phi^{(n)}},\seq{A^{(n)}},\left(\cX^{(n)}\right)_{n\geq 0},\left(V_n\right)_{n\geq 0},\seq{I^{(n)}},E,i,t\right)
	\end{align*} 
	where
	\begin{align*}
		\Phi^{(n)}=\set{\phi_e^{(n)}:X_{t(e)}^{(n)}\to X_{i(e)}^{(n-1)}}_{e\in I^{(n)}},
	\end{align*} 
	such that the following hold:
	\begin{enumerate}
		\item (Open Set Condition) $\phi_a^{(n)}(\intr{X^{(n)}_{t(a)}})\intersect\phi_b^{(n)}(\intr{X^{(n)}_{t(b)}})=\emptyset$ for all $n\in\NN$, and $a\not=b\in I^{(n)}$.
		
		\item (Conformality) For each $v\in V_n$ and $n\in\NN$ there is an open and connected $W_v^{(n)}\bus X_v^{(n)}$ (independent of $j$) such that for each $j\in I^{(n)}$ with $t(j)=v$, the map $\phi_j^{(n)}$ extends to a $C^1$ conformal diffeomorphism of $W_v^{(n)}$ into $W_{i(j)}^{(n-1)}$. Moreover, we can assume that 
		\begin{align}\label{Wvn diam}
			\diam(W_v^{(n)})\leq2 \diam(X_v^{(n)})
		\end{align}
		for each $n\in\NN$ and each $v\in V_n$.
		
		\item (Uniform Contraction) There is a constant $\eta\in(0,1)$ such that 
		\begin{align*}
		\absval{D\phi_{\om|_j^{j+m}}(x)}\leq \eta^m
		\end{align*} for all sufficiently large $m\in\NN$, all $\om\in I^\infty$, all $j\geq 1$, and all $x\in X_{t(\om_{j+m})}^{(j+m)}$, where 
		\begin{align*}
			\om|_j^{j+m}:=\om_j\om_{j+1}\dots\om_{j+m}.
		\end{align*}
		In the sequel, for the ease of exposition, we assume that 
		\begin{align*}
			\absval{D\phi_i^{(j)}(x)}\leq \eta
		\end{align*}
		for all $j\in\NN$, $\om\in I^\infty$, and $x\in X^{(j)}_{t(\om_j)}$.
	
		\item (Bounded Distortion) There is $K\geq 1$ such that for all $m\in\NN$, for any $k\leq m$, for all $\om\in I^{k,m}$, 
		\begin{align*}
		\absval{D\phi_\om^{k,m}(x)}\leq K\absval{D\phi_\om^{k,m}(y)}
		\end{align*} 
		for all $x,y\in X_{t(\om)}^{(m)}$.
		
		\item (Geometry Condition) 	There exists $N\in\NN$ such that for all $n\in\NN$ and all $v\in V_n$ there exist $\Gm_1^{(n)},\dots\Gm_N^{(n)}\sub W_v^{(n)}$ such that each of the $\Gm_j^{(n)}$ are convex, and 
		\begin{align*}
		X_v^{(n)}\sub\union_{j=1}^N\Gm_j^{(n)}.
		\end{align*}
		We also suppose there exists $\vartheta>0$ such that for each $x\in X_v^{(n)}$ we have that
		\begin{align}\label{hyp: comp ball in W}
		B(x,\vartheta\cdot\diam(X_v^{(n)}))\sub W_v^{(n)}.
		\end{align}		
		
		\item (Uniform Cone Condition) There exist $\al,\gm>0$ with $\gm<\frac{\pi}{2}$ such that for every $n\in\NN$, every $v\in V_n$, and every $x\in \partial X_v^{(n)}$ there exists $u_x$ and an open cone 
		\begin{align*}
		\Con(x,u_x,\gm,\al\cdot\diam(X_v^{(n)}))\sub\intr{X_v^{(n)}}
		\end{align*} 
		with vertex $x$, direction vector $u_x$ with unit length, central angle of Lebesgue measure $\gm$, and altitude $\al\cdot\diam(X_v^{(n)})$ comparable to $\diam(X_v^{(n)})$. 
		
		\item (Diameter Condition) For each $n\in\NN$ we have
		\begin{align*}
		\limty{n}\frac{1}{n}\log \ol{d}_n=0 \spand \limty{n}\frac{1}{n}\sup_{k\geq 0}\log\frac{\ol{d}_{k+n}}{\ul{d}_k}=0,
		\end{align*}
		where 
		\begin{align*}
		\ul{d}_n=\min\set{\diam(X_v^{(n)}):v\in V_n}\spand\ol{d}_n=\max\set{\diam(X_v^{(n)}):v\in V_n}.
		\end{align*}
		Furthermore, we assume that 
		\begin{align*}
		\limty{n}\frac{1}{n}\log \#V_n=0.
		\end{align*}

	\end{enumerate}
	A NCGDMS $\Phi$ is called \textit{stationary} if the sequence of sets $\cX^{(n)}$ is constant, i.e. if $\cX^{(n)}=\cX^{(m)}$ for all $n,m\in\NN$. In other words, the collection $\cX$ of compact connected spaces does not depend on the time $n$. To emphasize when a particular NCGDMS is not stationary, we will call that system \textit{non-stationary}. If the collections $V_n$ are singletons for every $n\in\NN$ and if the matrices $A^{(n)}$ contain only ones, i.e. every letter at time $n+1$ is allowed to follow every letter at time $n$, then we refer to the system $\Phi$ as a \textit{non-autonomous conformal iterated function system} (NCIFS). 
	
	A NCGDMS $\Phi$ is called \textit{finite} if the collections $\Phi^{(n)}$ are finite for each $n$, and \textit{infinite} otherwise. $\Phi$ is said to be \textit{uniformly finite} if there is a constant $M>0$ such that $\#I^{(n)}<M$ for each $n\in\NN$. In the sequel we will mainly work with finite NCGDMS in their full generality. However, in Section \ref{section Rem-Urb NCIFS} and parts of \ref{Ascending}, we extend the results of \cite{rempe-gillen_non-autonomous_2016} to particular classes of infinite NCIFS.
\end{definition}

\begin{remark}
	Notice that this definition generalizes the notion of conformal GDMS and NCIFS, in the sense of Rempe-Gillen and Urba\'nski. In particular, if $\Phi$ is a stationary NCGDMS such that the collections $\cX^{(n)}$ each contain the single space $X$, i.e. $\cX^{(n)}=\cX^{(m)}=\set{X}$ for all $n,m\in\NN$, then $\Phi$ is a NCIFS as described in \cite{rempe-gillen_non-autonomous_2016}. Rempe-Gillen and Urba\'nski described what we would call a stationary NCIFS. Unless stated otherwise, in this paper when we refer to a NCIFS we mean a NCGDMS, which is not necessarily non-stationary, such that for each $n\in\NN$ the collection $\cX^{(n)}$ of spaces at time $n$ is a singleton which may differ depending on the time $n$. 
	 
	Now if $\Phi$ is stationary and the collections of functions $\Phi^{(n)}=\Phi^{(m)}$ and the matrices $A^{(n)}=A^{(m)}$ for each $n\neq m$, then $\Phi$ is a conformal GDMS as described in \cite{mauldin_graph_2003}. If we wish to emphasize a system's lack of time dependence, we will call that system an \textit{autonomous} conformal GDMS. And of course, each of these different constructions is a generalization of the theory of conformal iterated function systems as discussed in \cite{mauldin_dimensions_1996}.
\end{remark}

\begin{remark}
	We note that the Bounded Distortion Property (BDP) is automatically satisfied whenever $d\geq 2$. The case $d=2$ follows from Koebe's Distortion Theorem, and the case $d\geq 3$ follows from Liouville's Theorem as whenever $d\geq 3$ each conformal map $\phi$ can be written as a composition of an inversion, a similarity, and a translation. In the case that $d=1$ the BDP needs to be checked. In the event that the system is comprised of similarities, then the BDP is automatically satisfied taking $K=1$. 
	
	An important consequence of the BDP is the following. Let $m<n\in\NN$ and let $\om\in I^n$ such that $\om=\al\bt$ with $\al\in I^m$ and $\bt\in I^{m+1,n}$ then we have
	\begin{align*}
		\norm{D\phi_\om}\leq\norm{D\phi_\al}\cdot\norm{D\phi_\bt}\leq K^2\norm{D\phi_\om}.
	\end{align*}
\end{remark}
\begin{remark}
	Although the Geometry Condition may seem overly prohibitive, we would argue instead that this condition is in fact quite mild. We point out that this condition is always satisfied in the case of any autonomous iterated function system or graph directed Markov system, finite or infinite. Furthermore, this condition is satisfied if the set $V$ of vertices is finite. We also note that this condition is necessary in order to obtain the following required dynamical assumption.
\end{remark}
\begin{proposition}
	There exists $C>0$ such that for all $m\leq n\in\NN$, all $\om\in I^{m,n}$, and all $U\sub X_{t(\om)}^{(n)}$ we have 
	\begin{align}\label{MV like ineq}
		\diam(\phi_\om(U))\leq C\norm{D\phi_\om}\diam(U)
	\end{align}
	and 
	\begin{align}\label{MV inv ineq}
		\diam\left(\phi_\om(X_{t(\om)}^{(n)})\right)\geq C^{-1}\norm{D\phi_\om}\diam(X_{t(\om)}^{(n)}).
	\end{align}
\end{proposition}
\begin{proof}
	Let $x\in U$. Then either $U\sub B\left(x,\vartheta\diam(X_{t(\om)}^{(n)})\right)$ or not. 
	If so, let $V$ denote the convex hull of $U$. Then $\diam(V)=\diam(U)$ and $V\sub W_{t(\om)}^{(n)}$. Then 
	\begin{align*}
		\diam(\phi_\om(U))\leq\diam(\phi_\om(V))\leq\norm{D\phi_\om}\diam(V)=\norm{D\phi_\om}\diam(U).
	\end{align*}
	Taking $C_1=1$, suffices. Now, if $U\not\sub B\left(x,\vartheta\diam(X_{t(\om)}^{(n)})\right)$, then 
	\begin{align}\label{diam U vartheta ineq}
		\diam(U)\geq \frac{1}{2} \vartheta\diam(X_{t(\om)}^{(n)}).
	\end{align}
	The Geometry Condition and mean value inequality together allow us to write 
	\begin{align*}
		\diam(\phi_\om(U))\leq \diam\left(\phi_\om(X_{t(\om)}^{(n)})\right)\leq \sum_{j=1}^N\diam\left(\phi_\om(\Gm_j^{(n)})\right)\leq \sum_{j=1}^N\norm{D\phi_\om}\diam(\Gm_j^{(n)}).
	\end{align*}
	In view of \eqref{Wvn diam} and given that $\Gm_j^{(n)}\sub W_{t(\om)}^{(n)}$ for each $1\leq j\leq N$ we have 
	\begin{align*}
		\diam(\phi_\om(U))\leq N\norm{D\phi_\om}\diam(W_{t(\om)}^{(n)})\leq 2N\norm{D\phi_\om}\diam(X_{t(\om)}^{(n)}).
	\end{align*}
	Finally, applying \eqref{diam U vartheta ineq}, we see that
	\begin{align*}
		\diam(\phi_\om(U))\leq \frac{4N}{\vartheta}\norm{D\phi_\om}\diam(U).
	\end{align*}
	Taking $C_1=\max\set{1,4N/\vartheta}$ finishes the first assertion.
	
	Now for the second assertion, fix $x\in X_{t(\om)}^{(n)}$. Then
	\begin{align*}
		B_\om^{(n)}:=B\left(x,\vartheta\diam(X_{t(\om)}^{(n)})\right)\sub W_{t(\om)}^{(n)}.	
	\end{align*} 
	Let $R\geq 0$ be the maximal radius such that $B(\phi_\om(x),R)\sub \phi_\om(B_\om^{(n)})$. Then 
	the Bounded Distortion Property implies that 
	\begin{align*}
		\phi_\om^{-1}\left(B(\phi_\om(x),R)\right)\sub B\left(x, \norm{D(\phi_\om^{-1})}R\right)\sub B\left(x, K\norm{D\phi_\om}^{-1}R\right).
	\end{align*}
	Therefore, $K\norm{D\phi_\om}^{-1}R\geq \vartheta\diam (X_{t(\om)}^{(n)})$, as the maximality of $R$ would be contradicted if this were not the case. Hence we see that 
	\begin{align*}
		\phi_\om\left(B_\om^{(n)}\right)\sub B\left(\phi_\om(x),K^{-1}\norm{D\phi_\om}\vartheta\diam(X_{t(\om)}^{(n)})\right)
	\end{align*}
	for every $x\in X_{t(\om)}^{(n)}$. Thus taking $C_2=\vartheta K^{-1}$ finishes the second assertion as we have 
	\begin{align*}	
		\diam\left(\phi_\om(X_{t(\om)}^{(n)})\right)\geq C_2\norm{D\phi_\om}\diam(X_{t(\om)}^{(n)}). 
	\end{align*}
	Finally, taking $C\geq \max\set{C_1, C_2^{-1}}$ completes the proof.
\end{proof}
In particular, the previous proposition gives that 
\begin{align*}
	 C^{-1}\norm{D\phi_\om}\diam(X_{t(\om)}^{(n)})
	\leq \diam\left(\phi_\om(X_{t(\om)}^{(n)})\right)
	\leq C\norm{D \phi_\om}\diam(X_{t(\om)}^{(n)}),
\end{align*}
and in fact
\begin{align*}
C^{-1}\ul{d}_n\norm{D\phi_\om}
\leq \diam\left(\phi_\om(X_{t(\om)}^{(n)})\right)
\leq C\ol{d}_n\norm{D\phi_\om},
\end{align*}
for all $\om\in I^n$ and $n\in\NN$. 
\begin{remark}
	While the two diameter conditions, may seem quite technical, we point out that this is certainly true whenever there is some constant $T\geq 1$ such that for each $n\in\NN$ and each $v\in V_n$
	\begin{align*}
		T^{-1}\leq \diam(X_v^{(n)})\leq T.
	\end{align*}
	We also note that the two diameter conditions together imply that 
	\begin{align}\label{dn low sub exp bdd}
	\limty{n}\frac{1}{n}\log \ul{d}_n=0.
	\end{align}
	Moreover, the assumption that
	\begin{align}\label{V_n growth assumption}
		\limty{n}\frac{1}{n}\log\#V_n=0
	\end{align} 
	is quite reasonable given that Rempe-Gillen and Urba\'nski showed that Bowen's formula does not hold in general for a stationary NCIFS for which the alphabets grow in size at least exponentially. In fact we will see the same phenomena for NCGDMS, and since $\#V_n\leq \#I^{(n)}$ for each $n\in\NN$ we will see that our restriction \eqref{V_n growth assumption} is much weaker than our requirement of subexponential growth of the alphabets.
\end{remark}
\begin{remark}
	The Uniform Cone Condition assures us that the spaces don't become too wild as the time step $n$ goes to infinity. In particular, we know that the spaces $X_v^{(n)}$ are not being ``pinched" too much, and instead we have that the ratio of the inner and outer diameters is uniformly bounded in between two positive constants. We also point out that this is a natural generalization of the cone condition presented in \cite{mauldin_graph_2003} for autonomous graph directed Markov systems as well as the one presented in \cite{roy_random_2011} for random graph directed Markov systems.  

	As a consequence of the Bounded Distortion Property, Uniform Cone Condition and Geometry Condition we have that there is some $\al'\in(0,\al]$ and $\gm'\in(0,\gm]$ such that
	 \begin{align}
		 \phi_\om(X_{t(\om)}^{(n)})&\bus\Con\left(\phi_\om(x),D\phi_\om(x)u_x,\gm',\norm{D\phi_\om}\al'\diam(X_{t(\om)}^{(n)})\right)\nonumber\\
		 &\bus\Con\left(\phi_\om(x),D\phi_\om(x)u_x,\gm',C^{-1}\al'\diam(\phi_\om(X_{t(\om)}^{(n)}))\right)\label{cone containment}
	 \end{align}
	 for each $\om\in I^{m,n}$ and $x\in X_{t(\om)}^{(n)}$. See Theorem 4.1.7, and the preceding results, of \cite{mauldin_graph_2003} for the details in the autonomous setting. The proof is similar in our setting, however we note that \eqref{hyp: comp ball in W} is crucial in showing that the constant $K_5$, of Theorem 4.1.6 of \cite{mauldin_graph_2003} is finite.   
\end{remark}
	The following definition and proposition provides natural conditions under which the Uniform Cone Condition, and furthermore the Geometry Condition, is satisfied.
	\begin{definition}
		Given a compact connected space $X$, we respectively define the \textit{inner radius} to be the following
		\begin{align*}
		\inrad(X)=\sup\set{r:B(x,r)\sub X, x\in X}.
		\end{align*}
		We say that a set $X$ satisfies the \textit{radius ratio property} if there is some $T>0$ such that 
		\begin{align*}
		\frac{\inrad(X)}{\diam(X)}\geq T.
		\end{align*}
		We say that a collection $\cU$ of sets satisfies the \textit{uniform radius ratio property} (URR) if there exists a uniform constant $T>0$ such that 
		\begin{align*}
		\frac{\inrad(U)}{\diam(U)}\geq T
		\end{align*}
		for every $U\in\cU$. 		
	\end{definition} 
\begin{proposition}\label{prop: convex geom cone replacement}
	Suppose that the collection 
	\begin{align*}
		\sX:=\set{X_\al}_{\gm\in \Gm}
	\end{align*} satisfies the URR property. Furthermore, suppose that $X_\gm$ is convex for each $\gm\in\Gm$. Then the collection $\sX$ satisfies the Uniform Cone Condition.
\end{proposition}
\begin{proof}
	As $\sX$ satisfies the uniform radius ratio property, each space $Y\in\sX$ must contain a ball of radius $r:=\inrad(Y)$ centered at some point $c$. Let $x\in\partial Y$ and choose a point $b$ on the circle of radius $r$ centered at $c$ such that the line segment from $c$ to $b$ is perpendicular to the line segment from $x$ to $c$. Since $r/\diam(Y)\geq T$ we have    	
	\begin{align*}
		\absval{c-b}=r,\spand T\cdot\diam(Y)\leq r\leq\absval{c-x}\leq \diam(Y).
	\end{align*}
	Letting 
	\begin{align*}
		0<\gm:=\tan^{-1}\left(\frac{r}{\diam(Y)}\right)<\pi/2\spand\al:=T<1,
	\end{align*}
	we see that $Y$ must contain the cone $\Con(x,u_x,\gm,\al\cdot\diam(Y))$, where $u_x$ is the unit length vector in the direction of $c$ from $x$, which satisfies the Uniform Cone Condition as desired.
\end{proof}
As a consequence of the previous proposition we see that if the collection 
\begin{align*}
	\sX=\set{X_v^{(n)}:v\in V_n, n\in\NN}
\end{align*}
satisfies the URR property and if each $X_v^{(n)}$ is convex for each $v\in V_n$ and each $n\in\NN$, then the Uniform Cone Condition is satisfied.
 
The main object of interest in the study of IFS or their generalizations is the limit set. To that end, for each $n\in\NN$ and each $\om\in I^n$ we define the level sets 
\begin{align*}
	Y_\om:=\phi_\om(X^{(n)}_{t(\om)})\sub X_{i(\om)}^{(0)} \spand Y_n:=\bigcup_{\om\in I^n}Y_\om.
\end{align*}
Then for $n\geq m$ we have $Y_n\sub Y_m\sub \union_{v\in V_0}X^{(0)}_v$. As $V_n$ is finite for each $n$, we have that the sets $Y_\om$ and $Y_n$ are compact.
\begin{definition}
	The \textit{limit set} of a NCGDMS $\Phi$ is defined to be the set 
	\begin{align*}
		J_\Phi:=\intersect_{n=1}^\infty\union_{\om\in I^n}\phi_\om(X^{(n)}_{t(\om)}).
	\end{align*}
	Note that $J_\Phi$ is contained in the disjoint union $\bigsqcup_{v\in V_0}X_v^{(0)}$. Furthermore, if the alphabets $I^{(n)}$ are each finite, then the limit set $J_\Phi$ is compact. However, if any of the sets $I^{(n)}$ are infinite, then the limit set is in general not closed, and in fact its closure may have a much larger Hausdorff dimension. It is well known that there are examples of infinite autonomous IFS and GDMS with limit sets which are dense but have Hausdorff dimension equal to zero. See, for example, \cite{mauldin_dimensions_1996}, \cite{mauldin_graph_2003}.
\end{definition}
Given $\om\in I^\infty$, the sets $Y_{\om\rvert_n}$, where 
\begin{align*}
	\om\rvert_n=\om_1\om_2\dots\om_n,
\end{align*} 
forms a decreasing sequence of compact sets whose diameters go to zero, thanks in part to the Uniform Contraction Principle and Diameter Condition. In fact, we have that the diameters go to zero exponentially fast as 
the following lemma shows.
\begin{lemma}\label{lem: diam go to zero}
	Given $\om\in I^\infty$, the sequence $\diam(Y_{\om\rvert_n})$ converges to zero exponentially fast.
\end{lemma}
\begin{proof}
	Equation \eqref{MV like ineq} gives that 
	\begin{align*}
		\diam(Y_{\om\rvert_n})\leq C\norm{D\phi_{\om\rvert_n}}\diam(X_{t(\om)}^{(n)}).
	\end{align*}
	Applying the Uniform Contraction Principle and Diameter Condition we see
	\begin{align*}
		\diam(Y_{\om\rvert_n})\leq C\eta^n\ol{d}_n.
	\end{align*}
	Let $0<\ep<\log(1/\eta)$. Now since $1/n\log\ol{d}_n\to 0$, we then have that $\ol{d}_n\leq e^{\ep n}$ for all $n$ sufficiently large. Inserting this into our previous estimate we now have that 
	\begin{align*}
			\diam(Y_{\om\rvert_n})\leq C(\eta e^\ep)^n.
	\end{align*}
	However, the quantity on the right hand side goes to zero, which completes the proof.
\end{proof}
Given that the sets $Y_{\om\rvert_n}$ are compact, the previous lemma gives that $\bigcap_{n\geq 1}Y_{\om|_{n}}$ is a singleton, and we denote its only element by $\pi_\Phi(\om)$. This association defines a projection map 
\begin{align*}
	\pi:=\pi_\Phi:I^\infty\to\bigsqcup_{v\in V_0}X_v^{(0)}.
\end{align*}
We end this section with the following elementary proposition which states that the limit set may also be thought of as the image of the set of infinite admissible words under the coding map $\pi$. Its proof is analogous to that of Lemma 2.4 of \cite{rempe-gillen_non-autonomous_2016}, however we shall include it here for the sake of completeness. 
\begin{proposition}\label{coding map limit set }
	$\pi_\Phi(I^\infty)=J_\Phi$.
\end{proposition}
\begin{proof}
	Given that Lemma \ref{lem: diam go to zero} shows that the diameters of the sequence $Y_{\om\rvert_n}$ go to zero as $n$ goes to infinity, we have that the map $\pi$ is well defined. Furthermore the inclusion $\pi(I^\infty)\sub J_\Phi$ follows from the definition of $\pi$. Thus we have only to check the opposite inclusion. To that end, let $x\in J_\Phi$. If $\Phi$ is a finite system we obviously have that the set 
	\begin{align}\label{set om in In finite}
		F=\set{\om\in I^n:x\in Y_\om}
	\end{align}
	is finite, however if $\Phi$ is an infinite system the set given in \eqref{set om in In finite} is still finite. This follows from the Open Set Condition together with the Uniform Cone Condition.
	
	We now form a directed graph on $F$ by drawing an edge from $\om$ to $\ut$ if and only if $\ut=\om_1\dots\om_{\absval{\om}-1}$, that is $\ut$ is equal to $\om$ with the last symbol, $\om_{\absval{\om}}$ deleted. As this graph contains arbitrarily long directed paths originating at the the empty word, the degree of each vertex is finite. Applying K\"onig's lemma, we have an infinite word $\om$ such that $x\in Y_{\om\rvert_n}$ for each $n\in\NN$. As we have that $\pi(\om)=x$, the proof is now complete.
\end{proof}
\section{Subexponential Boundedness and Finite Primitivity}\label{Sec: Subexp}
In this section we define two properties which will be imperative in the sequel and investigate their properties and interconnections. 
\begin{definition}
	A NCGDMS is called \textit{subexponentially bounded} if 
	\begin{align*}
	\limit{n}{\infty}\frac{1}{n}\log\#I^{(n)}=0.
	\end{align*}
\end{definition}
Given that Bowen's formula holds for both infinite and finite autonomous GDMS and IFS, one would expect that this should be the case for non-autonomous systems as well, or at least certainly for finite non-autonomous systems. However, Rempe-Gillen and Urba\'nski showed in \cite{rempe-gillen_non-autonomous_2016} that the problem is not that simple. For finite systems, the growth rate of the alphabets $I^{(n)}$ is intricately related to the dimension of the limit set. In fact they give examples of finite systems, which are well behaved in all aspects other than growth rates of the $I^{(n)}$, for which Bowen's formula does not hold. As we will see, subexponential boundedness will be necessary for Bowen's formula to hold for a general finite NCGDMS without imposing other stronger assumptions.  
\begin{definition}
	We say that a NCGDMS $\Phi$ is \textit{finitely primitive} if there is a constant $p\in\NN$ such that for each $n\in\NN$, there is a finite set $\Lambda_n\sub I^{n+1,n+p}$ such that the following hold.
	\begin{enumerate}
		\item For all $a\in I^{(n)}$, $b\in I^{(n+p+1)}$ there is $\lm(a,b):=\lm\in \Lambda_n$ such that $a\lm b\in I^{n,n+p+1}$.
		\item There is some constant $Q>0$ such that for each $m\in\NN$ and for each $\lm\in\Lambda_m$ we have $Q\leq\norm{D\phi_\lm^{m+1,m+p}}$.
	\end{enumerate}
	If the system $\Phi$ is finitely primitive with respect to a particular value $p$, we may say that $\Phi$ is \textit{$p$-finitely primitive} if we wish to call attention to the value $p$.
\end{definition}
This property will be necessary to ensure that we are always able to continue extending a given string.
\begin{remark}
	If a system $\Phi$ is $p$-finitely primitive then for each $n\in\NN$ we have then that the entries of the matrix 
	\begin{align*}
	A_n^p:=\prod_{j=n}^{n+p-1}A^{(j)}
	\end{align*}
	are all positive. Furthermore, every NCIFS is $0$-finitely primitive, that is each letter at time $n+1$ is allowed to follow each letter at time $n$.
\end{remark}
One important difference between the non-autonomous graph directed systems and NCIFS is the condition of finite primitivity. In the case of NCIFS, as described in \cite{rempe-gillen_non-autonomous_2016}, each letter at time $n+1$ is allowed to follow each of the letters at time $n$, which ensures that each letter plays an equal role in forming the limit set. It is precisely the finite primitivity property that gives us the analogous notion for NCGDMS. We will now explore further properties of finite primitivity and their connections with subexponential boundedness, but first we establish some useful notation. 

For $\om\in I^n$ and $m\in\NN$ we let $L_\om^{n+1,n+m}$ denote the set of all words of length $m$ which are allowed to follow $\om$. More precisely, we have
\begin{align*}
	L_\om^{n+1,n+m}=\set{\gm\in I^{n+1,n+m}_{t(\om)}:\om\gm\in I^{n+m}}.
\end{align*} 
If $m=1$ or if $m=\infty$, we instead use the notation $L_\om^{(n+1)}$ or $L_\om^\infty$ respectively. Now for $m\in\NN$ define 
\begin{align*}
	\ul{G}_n^m:=\min\set{\#L_\om^{n+1,n+m}: \om\in I^n} \spand \ol{G}_n^m:=\max\set{\#L_\om^{n+1,n+m}: \om\in I^n},	
\end{align*}
and let 
\begin{align*}
	\xi_n^m:=\frac{\ol{G}_n^m}{\ul{G}_n^m}.
\end{align*} 
The quantities $\ol{G}_n^m$ and $\ul{G}_n^m$ respectively represent the greatest and least number of possible words of length $m$ that are allowed to follow any given letter in $I^{(n)}$. If $m=1$, we adopt the notation $\ul{G}_n$, $\ol{G}_n$, and $\xi_n$ instead. Since we are interested in only words with infinite extensions, we may assume that each letter has at least one follower, and so we have that $\ul{G}_n^m\geq 1$, which subsequently implies that 
\begin{align}\label{xi less than G up }
	1\leq \xi_n^m\leq \ol{G}_n^m
\end{align}
for all $n, m\in\NN$. 
\begin{observation}\label{obs:AlphaIneq1}
	The inequalities
	\begin{align*}
		\ul{G}_n\leq\ol{G}_n\leq\#I^{(n+1)}\leq \ol{G}_n\cdot\#I^{(n)}\spand
		\ol{G}_n^m\leq\ol{G}_n\cdots\ol{G}_{n+m-1}
	\end{align*}
	 are clear, and imply that 
	\begin{align*}
		\#I^{(n+m+1)}\leq \#I^{n,n+m+1}\leq \#I^{(n)}\cdot\ol{G}_n\cdots\ol{G}_{n+m}.	
	\end{align*}
\end{observation}

\begin{observation}\label{obs:AlphaIneq2}
	If a system is $p$-finitely primitive, then by definition, for all $a\in I^{(n)}$ and $b\in I^{(n+p+1)}$ there is some $\lm\in I^{n+1,n+p}$ such that $a\lm b$ is admissible. Thus we have that 
	\begin{align*}
	\#I^{n,n+p+1}\geq \#I^{(n)}\cdot\#I^{(n+p+1)}.
	\end{align*} 
\end{observation}

\begin{observation} Since $p$-finite primitivity gives us that for each $a\in I^{(n)}$ and $b\in I^{(n+p+1)}$ there is a word of length $p$ from $a$ to $b$, we know that each $a\in I^{(n)}$ must have at least as many $p+1$ length extensions as there are letters in $I^{(n+p+1)}$, i.e. 
\begin{align*}
	\#L_a^{n+1,n+p+1}\geq \#I^{(n+p+1)}
\end{align*}
for each $a\in I^{(n)}$. Thus we have the inequality
	\begin{align}\label{ineq: obs 3.6}
	\#I^{(n+p+1)}\leq \ul{G}_n^{p+1}\leq \ol{G}_n^{p+1}\leq\#I^{n,n+p+1}.
	\end{align}
\end{observation}

Combining all of the previous inequalities, we have the summary
\begin{align*}
\ul{G}_{n+p}\leq\ol{G}_{n+p}\leq\#I^{(n+p+1)}\leq \ul{G}_n^{p+1}\leq \ol{G}_n^{p+1}\leq\#I^{n,n+p+1}\leq \#I^{(n)}\cdot\ol{G}_n\cdots\ol{G}_{n+p}.
\end{align*}
The following result follows immediately from the first inequality of Observation \ref{obs:AlphaIneq1}.
\begin{corollary}\label{cor:subexp bdd implies sub G_n bdd}
	If a system $\Phi$ is subexponentially bounded then $\lim_{n\to\infty}\frac{1}{n}\log\ol{G_n}=0$.
\end{corollary}
The next lemma supplies a sort of converse to the previous corollary. 
\begin{lemma}\label{lem:FinPrimPlusGnSubexp->SubexpBdd}
	If a system is $p$-finitely primitive and $\limsupty{n}\frac{1}{n}\log\ol{G}_n\leq L$ then
	\begin{align*}
	\limsupty{n}\frac{1}{n}\log\#I^{(n)}\leq Lp.
	\end{align*}
\end{lemma}
\begin{proof}
	In light of the previous observation, we see that it suffices to show that $\ol{G}_n^p$ also grows at most exponentially with $n$. For $\ep>0$, Observation \ref{obs:AlphaIneq1} gives us that 
	\begin{align*}
	\log\ol{G}_n^p&\leq\sum_{i=n}^{n+p-1}\log\ol{G}_i\leq (L+\ep)\cdot(n+n+1+\dots+n+p-1)\\
	&=(L+\ep)\cdot[np+\frac{p(p-1)}{2}]
	\end{align*}
	for $n$ sufficiently large. Thus we see
	\begin{align*}
	\frac{1}{n}\log\ol{G}_n^p\leq Lp+\ep p+\frac{(L+\ep)\cdot p(p-1)}{2n}.
	\end{align*}
	As $\ep>0$ was arbitrary, we have that 
	\begin{align*}
		\limsupty{n}\frac{1}{n}\log\ol{G}_n^p\leq Lp
	\end{align*} 
	as desired. Since $\ol{G}_n^{p+1}$ grows subexponentially and given \eqref{ineq: obs 3.6}, we must have that 
	\begin{align*}
	\limsupty{n}\frac{1}{n}\log\#I^{(n)}\leq Lp,
	\end{align*}
	which finishes the proof.
	
\end{proof}
The previous lemma now immediately implies the following corollary.
\begin{corollary}\label{cor: fin prim + G implies subexp}
	If a NCGDMS $\Phi$ is finitely primitive and $\limty{n}\frac{1}{n}\log\ol{G}_n=0$, then $\Phi$ is subexponentially bounded.
\end{corollary}
Together Corollaries \ref{cor:subexp bdd implies sub G_n bdd} and \ref{cor: fin prim + G implies subexp} give that for a finitely primitive NCGDMS $\Phi$, subexponential boundedness is equivalent to the condition that
\begin{align*}
	\limty{n}\frac{1}{n}\log\ol{G}_n=0.
\end{align*}
In the sequel we will see that this condition plays the analogous role in the proof of Bowen's formula for finite NCGDMS as subexponential boundedness did for the respective proofs in the setting of finite NCIFS of \cite{rempe-gillen_non-autonomous_2016}. 
The following proposition shows that, without loss of generality, we may assume that any finitely primitive NCGDMS is in fact $1$-finitely primitive, i.e. given two letters $a\in I^{(n)}$ and $b\in I^{(n+2)}$ there is some $c\in \Lm_n\sub I^{(n+1)}$ such that the word $acb$ is admissible for every $n\geq 1$.  
\begin{proposition}\label{prop:1 fin prim}
	If $\Phi$ is a $p$-finitely primitive NCGDMS, then there exists a subsystem $\hat{\Phi}$ of $\Phi$ that is $1$-finitely primitive such that $J_{\hat{\Phi}}=J_\Phi$. 
\end{proposition}
\begin{proof}
	First we note that if $\Phi$ is $p$-finitely primitive then for fixed $n\in\NN$ and $a\in I^{(n)}$ there is some $\lm(a,b)\in\Lambda_n\sub I^{n+1,n+p}$ such that $a\lm(a,b) b$ is admissible for each $b\in I^{(n+p+1)}$. With that in mind, we create new alphabets
	\begin{align*}
	\hat{I}^{(n)}:=I^{(n-1)p+1,np} \quad \text{ for } n\geq 1.
	\end{align*}	
	Making the appropriate modifications by taking new incidence matrices
	\begin{align*}
	\hat{A}^{(n)}:\hat{I}^{(n)}\times \hat{I}^{(n+1)}\to \set{0,1}	
	\end{align*} 
	given by 
	\begin{align*}
	\hat{A}^{(n)}_{\al\bt}=A^{(np)}_{t(\al)i(\bt)}, 
	\end{align*}
	and letting 
	\begin{align*}
	\hat{\Phi}^{(n)}=\set{\hat{\phi}_e^{(n)}:=\phi_{t(e)}^{(np)}\circ\dots\circ\phi_{i(e)}^{((n-1)p+1)}:X_{t(e)}^{(np)}\to X_{i(e)}^{((n-1)p+1)}}_{e\in \hat{I}^{(n)}},
	\end{align*}
	$\hat{\Phi}$ becomes a NCGDMS subsystem of $\Phi$ such that $J_{\hat{\Phi}}=J_\Phi$. 
	
	Moreover $\hat{\Phi}$ is $1$-finitely primitive. To see this, let $\al\in \hat{I}^{(n)}$ and $\bt\in\hat{I}^{(n+2)}$. Setting $a:=t(\al)\in I^{(np)}$ and $b:=i(\bt)\in I^{((n+1)p+1)}$, then by the $p$-finite primitivity of $\Phi$ there is some 
	\begin{align*}
	\lm(a,b)\in\Lm_{np}\sub I^{np+1,(n+1)p}
	\end{align*} 
	such that $a\lm(a,b)b$ is admissible in $\Phi$. Taking
	\begin{align*}
	\gm:=\lm(a,b)\in I^{np+1,(n+1)p}=\hat{I}^{(n+1)}
	\end{align*}     
	we have that $\al\gm\bt$ is admissible in $\hat{\Phi}$. Now letting 
	\begin{align*}
	\hat{\Lm}_n:=\Lm_{np}
	\end{align*}
	and $\hat{Q}:=Q$, from the definition of finite primitivity, shows that $\hat{\Phi}$ is $1$-finitely primitive.
	
\end{proof}
\begin{remark}
	If we take 
	\begin{align*}
	\hat{\ol{G}}_n=\ol{G}_n^p,
	\end{align*}
	in the above proof, then in light of the proof of Lemma \ref{lem:FinPrimPlusGnSubexp->SubexpBdd} and Corollary \ref{cor: fin prim + G implies subexp} taken together, we note that that $\hat{\Phi}$ is in fact subexponentially bounded if $\Phi$ is subexponentially bounded.
\end{remark}

\begin{remark}
	Notice that we are unable to say that any letter at time $n+1$ is allowed to follow any letter at time $n$ as it may well be the case that the codomain and domain of the respective associated functions may not be equal.
\end{remark}

\begin{remark}
	Despite the simplifying nature of Proposition \ref{prop:1 fin prim}, in the sequel we will generally still work with $p$-finitely primitive systems as doing so will not muddy the proofs, but in fact help us to keep better track of certain indices. 
\end{remark}

\section{Pressure and Bowen Dimension}\label{Sec: Pressure}
In this section we introduce the pressure and partition functions as well as the Bowen dimension. 
\begin{definition}
For each $m,n\in\NN$ and each $t\in [0,d]$ we define the following partition functions:
\begin{align*}
	Z_{(n)}(t):=\sum_{\om\in I^{(n)}}\norm{D\phi_\om^{(n)}}^t \spand
	Z_{m,n}(t):=\sum_{\om\in I^{m,n}}\norm{D\phi_\om^{m,n}}^t.
\end{align*}
If $m=1$ then we denote $Z_{1,n}(t)$ simply by $Z_n(t)$. Now we define the lower and upper pressure functions respectively to be 
\begin{align*}
\ul{P}(t):=\ul{P}^\Phi(t):=\liminfty{n} \frac{1}{n} \log Z_n(t)
\spand \ol{P}(t):=\ol{P}^\Phi(t):=\limsup_{n\to\infty} \frac{1}{n} \log Z_n(t).
\end{align*}
\end{definition}
\begin{remark}
	Notice that we can define $Z_n(t)$ differently if we sum over all letters stemming from a particular vertex, in other words for $n,k\in\NN$ we can write 
	\begin{align*}
		Z_{n+k;v}(t)=\ds\sum_{\om\in I^n_v}\sum_{\ut\in L^{n+1,n+k}_\om}\norm{D\phi_{\om\ut}}^t
		\spand
		Z_n(t)=\sum_{v\in V_n}Z_{n;v}(t),
	\end{align*}
	where 
	\begin{align*}
		Z_{n;v}(t):=\ds\sum_{\om\in I^n_v}\norm{D\phi_\om}^t.
	\end{align*}
\end{remark}
We now come to the following definition which will be our main focus in the sequel. 
\begin{definition}
	The \textit{Bowen dimension} of a given NCGDMS $\Phi$ is defined to be the number
	\begin{align*}
	B_\Phi&:=\sup\set{t\geq 0:\ul{P}(t)\geq 0}=\inf\set{t\geq 0: \ul{P}(t) \leq 0}=\sup\set{t\geq 0:Z_n(t)\to\infty}.	
	\end{align*}
If $\HD(J_\Phi)=B_\Phi$ we say that \textit{Bowen's formula holds}.	
\end{definition}
\begin{remark}
	We remark that unlike the case of autonomous IFS and autonomous GDMS which are strongly regular, the pressure function may not actually ever be equal to zero, however we are still interested in the exact moment when the pressure function changes from being positive to negative. 
\end{remark}
As we shall see shortly, the Bowen dimension serves as an upper bound for the Hausdorff dimension the limit set of a NCGDMS. Following the ideas of \cite{mauldin_dimensions_1996}, we can use the pressure and partition functions to find similar lower bounds for the Hausdorff dimension, which while quite often are strict lower bounds, are quite simple to calculate. 
We define the following two numbers
\begin{align*}
	\theta_N&=\inf\set{t\geq 0: Z_{(n)}(t)<\infty \text{ for all }n\in\NN}\spand
	\theta_\Phi=\inf\set{t\geq 0:\ul{P}(t)<\infty}.
\end{align*} 
Then by definition we have that 
\begin{align}\label{thetaN lower bound ineq}
	0\leq\theta_N\leq \theta_\Phi\leq B_\Phi\leq d.
\end{align}
It is also worth mentioning that $\theta_N$, while providing a weaker lower bound, is typically easier to calculate than $\theta_\Phi$. In Section \ref{Sec: Examples} we will use these lower bounds to estimate the dimension of the Julia set for non-autonomous affine perturbations of an elliptic function. 

The following lemma assures us that the non-autonomous pressure functions still have the properties one might expect. 
\begin{lemma}
The lower pressure function is strictly decreasing on the interval $[\theta_\Phi,d]$ and infinite for all $0\leq t<\theta_\Phi$.
\end{lemma}
\begin{proof}
The second claim follows from the definition of $\theta_\Phi$, so it suffices to consider $t>\theta_\Phi$ and let $\ep>0$. As $t>\theta_\Phi$ we have that $\limsupty{n}Z_n(t)<\infty$. Then  
\begin{align*}
	Z_n(t+\ep)=\ds\sum_{\om\in I^n}\norm{D\phi_\om}^{t+\ep}\leq \sum_{\om\in I^n}\eta^{n\ep}\norm{D\phi_\om}^t=\eta^{n\ep}Z_n(t).
\end{align*}
Thus we have 
\begin{align*}
	\ul{P}(t+\ep)=\ds\liminf_{n\to\infty} \frac{1}{n} \log Z_n(t+\ep)\leq \liminf_{n\to\infty} \frac{1}{n} \log(\eta^{n\ep}Z_n(t))=\ep\log\eta + \ul{P}(t)\leq \ul{P}(t),
\end{align*}
which finished the proof.
\end{proof}

\begin{remark}
Similar statements and proofs show that the upper pressure function as well as the upper and lower pressure functions for a given vertex $v\in V_0$.
\end{remark}
In the next lemma we show that the Bowen dimension of a NCGDMS $\Phi$ is always an upper bound of the Hausdorff dimension of the limit set $J_\Phi$.
\begin{lemma}\label{lem:HDlessB}
If $\Phi$ is a NCGDMS then
\begin{align*}
	\HD(J_\Phi)\leq B_\Phi.		
\end{align*} 
\end{lemma}
\begin{proof}
Let $\scr{U}_n=\set{\phi_\om(X_{t(\om)}^{(n)}):\om\in I^n}$. Then $\sU_n$ is a cover of $J_\Phi$ such that $\diam(\scr{U}_n)\leq \eta^n$. 
Fix $t>B_\Phi$. Then $\ul{P}(t)<0$. Calculating the $t$-dimensional Hausdorff measure
we have 
\begin{align*}
	\sum_{\om\in I^n}\diam^t(\phi_\om(X_{t(\om)}^{(n)}))
	\leq \sum_{\om \in I^n}\left(C\cdot \norm{D\phi_\om}\cdot\diam(X_{t(\om)}^{(n)})\right)^t
	\leq (C\ol{d}_n)^tZ_n(t).
\end{align*}
We now let $\ep>0$ sufficiently small such that $(\ul{P}(t)+\ep(t+1))<0$, and recall that 
\begin{align*}
\limty{n}\frac{1}{n}\log\ol{d}_n=0.
\end{align*}
By taking $n$ sufficiently large, we see
\begin{align}\label{ineq: butter knife}
	H^t_\dl(J_\Phi)\leq \liminf_{n\to\infty}\sum_{\om\in I^n}\diam^t(Y_\om)\leq \liminfty{n}(C\ol{d}_n)^tZ_n(t)
	\leq\liminfty{n}e^{n(\ul{P}(t)+\ep(t+1))}=0.
\end{align}
Thus, letting $\dl\to 0$ we have  $H^t(J_\Phi)=0$. As this is true for each $t\geq B_\Phi$, we have that $\HD(J_\Phi)\leq B_\Phi$.
\end{proof}
Lemma \ref{lem:HDlessB} shows that half of Bowen's formula always holds. The next few sections will be devoted to proving the other half. As in the case of the theory of stationary NCIFS developed in \cite{rempe-gillen_non-autonomous_2016}, we cannot expect that the other half of Bowen's formula, i.e. $B_\Phi\leq\HD(J_\Phi)$, will hold in the generality of Lemma \ref{lem:HDlessB}.

\section{General Lower Bound}\label{Sec: Gen Low Bdd}
In this section we give a general lower bound for the Hausdorff dimension of a general NCGDMS. This rather technical theorem will be our main tool in establishing Bowen's formula for certain classes of non-autonomous systems in the sequel. 

Bowen's formula is based on the idea that the level sets $Y_n=\union_{\om\in I^n}\phi_\om(X_{t(\om)}^{(n)})$ are the optimal coverings of the limit set $J_\Phi$ in the sense of Hausdorff measure. Unlike the autonomous case, this is not necessarily true for the non-autonomous case without additional assumptions. In order to estimate a lower bound of the Hausdorff dimension of the limit set we first let
\begin{align*}
	\ul{c}_n&:=\min_{a\in I^{(n)}}\norm{D\phi_a^{(n)}} \spand \ol{c}_n:=\max_{a\in I^{(n)}}\norm{D\phi_a^{(n)}}.
\end{align*}
for each $n\in\NN$. Note that $0<\ul{c}_n\leq\ol{c}_n<1$ by the uniform contraction principle. The Open Set Condition and Uniform Cone Condition together give that the volume of the set $\union_{a\in I^{(n)}}\phi_{a}^{(n)}(X_{t(a)}^{(n)})$ is larger, up to a constant multiple, than
\begin{align*}
	C\#I^{(n)}\cdot\ul{c}_n\cdot\ul{d}_n.
\end{align*} 
In particular, we have
\begin{align*}
	\diam\left(\union_{a\in I^{(n)}}\phi_a^{(n)}(X_{t(a)}^{(n)})\right)&\geq \const\cdot\sqrt[d]{\#I^{(n)}}\cdot\ul{c}_{n}\cdot\ul{d_n}\\
&\geq \const\cdot(G_{n-1})^{\frac{1}{d}}\cdot\ul{c}_n\cdot \ul{d}_n.
\end{align*}
Thus for each $n\in\NN$ we let 
\begin{align*}
	\widetilde{Z}_{n+1}(t)&:=Z_{n}(t)\cdot\left(\ul{G}_{n}\right)^{\frac{t}{d}}\cdot\ul{c}_{n+1}^t\cdot (\ul{d}_{n+1})^t.
\end{align*}
Considering the definition of the Bowen dimension with $Z_n(t)$ replaced with $\widetilde{Z}_n(t)$, one might quite reasonably expect that $\HD(J_\Phi)\geq t$ whenever $\liminf\widetilde{Z}_n(t)>0$. The next theorem will show that this is indeed the case whenever the following constants are well behaved. Let
\begin{align*}
\rho_n&:=\sup_{a,b\in I^{(n)}}\frac{\norm{D\phi_a^{(n)}}}{\norm{D\phi_b^{(n)}}}=\frac{\ol{c}_n}{\ul{c}_n}\geq 1,
\end{align*}
and let 
\begin{align}\label{assumption 2}
\dl:=\limsup_{n\to\infty}\frac{1}{n}\log\ol{G}_n.
\end{align}
Finally, defining
\begin{align}\label{def: kappa(t)}
	\kp(t):=\liminf_{n\to\infty} \frac{1}{n}\log\frac{\widetilde{Z}_{n}(t)}{1+\log\max_{j\leq n+1}\set{\rho_j}+\sup_{k\geq 0}\log(\ol{d}_{n+k}/\ul{d}_n)},
\end{align}
we have now established all of the machinery necessary to prove our main technical theorem. Its proof relies on the use of the mass distribution principle, see, for example, \cite{falconer_fractal_1990}.
\begin{theorem}\label{thm:LBforHDJv}
	Suppose $\Phi$ is a $p$-finitely primitive finite NCGDMS. If $t>0$ such that 
	\begin{align*}
	\dl<\frac{\kp(t)}{p^2+p+1},
	\end{align*}
	we have $\HD(J_\Phi)>t$.
\end{theorem}
\begin{proof}
	Choose $0<\dl',\kp'$ such that 
	\begin{align*}
	0\leq\dl<\dl'<\frac{\kp'}{p^2+p+1}<\frac{\kp(t)}{p^2+p+1}.
	\end{align*} 
	Let $m_j$ be a measure supported on $Y_j$ defined such that for $\om\in I^j$ 
	\begin{align*}
	m_j(Y_\om)=m_j\left(\phi_\om(X^{(j)}_{t(\om)})\right)=\frac{\norm{D\phi_\om}^t}{Z_{j}(t)}\cdot
	\end{align*}
	
	Let $n+p<j$, and let $\om\in I^n$. As we wish to estimate $m_j(Y_\om)$, we first note that finite primitivity allows us to write 
	\begin{align*}
	Z_{j}(t)&=\sum_{\om\in I^j}\norm{D\phi_\om}^t
	\geq\sum_{\mathclap{\substack{\beta\in I^n\\ \lm\in\Lambda_n\text{, }\al\in I^{n+p+1,j}\\ \beta\lm\al\in I^j}}}\norm{D\phi_{\beta\lm\al}}^t\geq K^{-4t}\sum_{\mathclap{\substack{\beta\in I^n\\ \lm\in\Lambda_n\text{, }\al\in I^{n+p+1,j}\\ \beta\lm\al\in I^j}}}\norm{D\phi_\beta}^t\cdot\norm{D\phi_\lm^{n+1,n+p}}^t\cdot\norm{D\phi_\al^{n+p+1,j}}^t\\
	&\geq K^{-4t}\sum_{\mathclap{\substack{\beta\in I^n\\ \al\in I^{n+p+1,j}\\ }}}\norm{D\phi_\beta}^t\cdot Q^{t}\cdot\norm{D\phi_\al^{n+p+1,j}}^t
	= K^{-4t}Q^{t}\sum_{\beta\in I^n}\norm{D\phi_\beta}^t\sum_{\al\in I^{n+p+1,j}}\norm{D\phi_\al^{n+p+1,j}}^t\\
	&=\const\cdot Z_{n}(t)\cdot Z_{n+p+1,j}(t).
	\end{align*}
	So in particular we have 
	\begin{align}\label{eqn:pq1lessthanfraction}
	Z_{n+p+1,j}(t)\leq\const\frac{Z_{j}(t)}{Z_{n}(t)}.
	\end{align}
	Now 
	\begin{align*}
	m_j(Y_\om)&=m_j\left(\union_{\gm\in L^{n+1,j}_\om}\phi_{\om\gm}(X^{(j)}_{t(\om\gm)})\right)\\
	&=\sum_{\gm\in L^{n+1,j}_\om}m_j\left(\phi_{\om\gm}(X^{(j)}_{t(\om\gm)})\right)
	\end{align*}
	since $\phi_{\om\gm}(X_{t(\om\gm)})\intersect\phi_{\om\gm'}(X_{t(\om\gm')})=\emptyset$ for each $\gm\neq\gm'\in L^{n+1,j}_\om$.
	
	Then, applying the Bounded Distortion Property we have
	\begin{align*}
	\sum_{\gm\in L^{n+1,j}_\om}m_j(Y_{\om\gm})
	&=\sum_{\gm\in L^{n+1,j}_\om}\frac{\norm{D\phi_{\om\gm}}^t}{Z_{j}(t)}
	\leq \norm{D\phi_\om}^t\cdot\sum_{\gm\in L^{n+1,j}_\om}\frac{\norm{D\phi^{n+1,j}_\gm}^t}{Z_{j}(t)}\\
	&\leq \frac{\norm{D\phi_\om}^t}{Z_{j}(t)}\cdot\sum_{\gm\in I^{n+1,j}}\norm{D\phi_\gm^{n+1,j}}^t
	=\frac{\norm{D\phi_\om}^t}{Z_{j}(t)}\cdot Z_{n+1,j}(t).
	\end{align*}
	Thus we have 
	\begin{equation}\label{eqn:ineq}
	\sum_{\gm\in L^{n+1,j}_\om}m_j(Y_{\om\gm})
	\leq\frac{\norm{D\phi_\om}^t}{Z_{j}(t)}\cdot Z_{n+1,j}(t).
	\end{equation}
	Invoking our prior observation and combining with equation \eqref{eqn:pq1lessthanfraction} we have 
	\begin{align*}
	m_j(Y_\om)
	&\leq\frac{\norm{D\phi_\om}^t}{Z_{j}(t)}\cdot Z_{n+1,j}(t)
	=\frac{\norm{D\phi_\om}^t}{Z_{j}(t)}\cdot\frac{Z_{n+1,j}(t)}{Z_{n+p+1,j}(t)}\cdot Z_{n+p+1,j}(t)\\
	&\leq \const\cdot \frac{\norm{D\phi_\om}^t}{Z_{j}(t)}\cdot\frac{Z_{n+1,j}(t)}{Z_{n+p+1,j}(t)}\cdot\frac{Z_{j}(t)}{Z_{n}(t)}\\
	&=\const\cdot \frac{\norm{D\phi_\om}^t}{Z_{n}(t)}\cdot\frac{Z_{n+1,j}(t)}{Z_{n+p+1,j}(t)}.
	\end{align*}
	Noticing that 
	\begin{align*}
	Z_{n+1,j}(t)
	&=\sum_{\beta\in I^{n+1,n+p}}\sum_{\gm\in L^{n+p+1,j}_\beta}\norm{D\phi_{\beta\gm}}^t
	\leq \sum_{\beta\in I^{n+1,n+p}}\sum_{\gm\in L^{n+p+1,j}_\beta}\norm{D\phi_{\beta}}^t\norm{D\phi_\gm}^t\\
	&\leq\left(\sum_{\beta\in I^{n+1,n+p}}\norm{D\phi_\beta}^t\right)\cdot\left(\sum_{\gm\in I^{n+p+1,j}}\norm{D\phi_\gm}^t\right)\\
	&=Z_{n+1,n+p}(t)\cdot Z_{n+p+1,j}(t).
	\end{align*}
	Then we are able to estimate the ratio 
	\begin{align*}
	\frac{Z_{n+1,j}(t)}{Z_{n+p+1,j}(t)}
	&\leq\frac{Z_{n+1,n+p}(t)\cdot Z_{n+p+1,j(t)}}{Z_{n+p+1,j}(t)}=Z_{n+1,n+p}(t).
	\end{align*}
	Recall that Lemma \ref{lem:FinPrimPlusGnSubexp->SubexpBdd} implies that if $\limsup_{n\to \infty}\frac{\log\ol{G}_n}{n}<\dl'<\infty$ then 
	\begin{align*}
	\limsup_{n\to \infty}\frac{\log\# I^{(n)}}{n}<p\cdot\dl'.
	\end{align*} 
	Since $\ol{c}_n<1$, for sufficiently large $n$ we have 
	\begin{align*}
	\sum_{a\in I^{(n)}}\norm{D\phi_a^{(n)}}^t\leq\#I^{(n)}\cdot\ol{c}_n^t\leq\#I^{(n)}\leq e^{p\cdot\dl' n}.
	\end{align*}
	Hence we see that 
	\begin{align*}
	Z_{n+1,n+p}(t)\leq e^{p\cdot\dl' n}\cdot e^{p\cdot\dl' (n+1)}\dots e^{p\cdot\dl' (n+p)}=e^{p\cdot\dl'(n+n+1+\cdots+n+p)}=e^{p\cdot\dl' \cdot n(p+1)}\cdot e^{\frac{1}{2}\dl'\cdot p^2(p+1)}.
	\end{align*}
	Since $p$ is fixed, and taking $\ep=\dl'\cdot p(p+1)$ we have 
	\begin{align*}
	Z_{n+1,n+p}(t)\leq \const\cdot e^{\ep n},
	\end{align*}
	and consequently we are left with
	\begin{align}\label{mj of Yom}
	m_j(Y_\om)\leq\const\cdot\frac{e^{\ep n}\norm{D\phi_\om}^t}{Z_{n}(t)}.
	\end{align}
	Now let $r>0$ and let $B\sub\RR^d$ be a ball of radius $r$. Let $W\sub I^*$ be the set of finite words such that each $\om\in W$ satisfies the following conditions:
	\begin{itemize}
		\item $B\intersect Y_\om\nonempty$,
		\item $\diam(Y_\om)\geq r$,
		\item $\diam(Y_{\om a})<r$ for some $a\in L^{(\absval{\om}+1)}_\om$.
	\end{itemize}
	Now let $W'\sub W$ be the set of all words in $W$ which are not extensions of some other word in $W$. Then for $\om\not=\gm\in W'$, the sets $Y_\om$ and $Y_\gm$ are disjoint and have diameters at least as great as $r$.	It follows from the Bounded Distortion Principle and Uniform Cone Condition that each set $Y_\om$ contains a cone based at a point of $B$ with diameter comparable to $r$. Indeed, from \eqref{cone containment}, we see that 
	\begin{align*}
		Y_\om\bus\Con\left(\phi_\om(x),D\phi_\om(x)u_x,\gm',C^{-1}\al'\diam(Y_\om)\right)\bus\Con\left(\phi_\om(x),D\phi_\om(x)u_x,\gm',C^{-1}\al'r\right)
	\end{align*}
	for each $\om\in W'$ and $x\in X_{t(\om)}^{(\absval{\om})}$. Considering the total volume of these cones, we see that there must be some constant $P>0$, independent of $r$, such that at most $P$ of the sets $Y_\om$, for $\om\in W'$ can be pairwise disjoint. In particular, we must have that $\#W'\leq P$ and for each $k$, the number of words in $W$ of length $k$ must also bounded by $P$. In other words, $\#(I^k\intersect W)\leq P$.
	
	Let $\ut\in W'$ and set $k=\absval{\ut}$. For each $m\geq 0$ define the set 
	\begin{align*}
	W_\ut(m)=\set{\om\in I^{k+m}\intersect W:\ut=\om|_k},
	\end{align*}
	where $\om\rvert_k:=\om_1\om_2\dots\om_k$. 	Let $\om\in W_\ut(m)$. Then 
	\begin{align}
	\diam(Y_{\om|_{k+1}})&\leq C\norm{D\phi_{\om|_{k+1}}}\cdot\diam(X^{(k+1)}_{t(\om_{k+1})})\nonumber\\
	&\leq C\norm{D\phi_{\om|_{k}}}\cdot\norm{D\phi_{\om_{k+1}}^{(k+1)}}\cdot\diam(X^{(k+1)}_{t(\om_{k+1})})\nonumber\\
	&\leq C^2\frac{\diam(Y_{\om|_k})}{\diam(X^{(k)}_{t(\om_k)})}\cdot \norm{D\phi^{(k+1)}_{\om_{k+1}}}\cdot\diam(X^{(k+1)}_{t(\om_{k+1})})\nonumber\\
	&=C^2\diam(Y_{\om|_k})\cdot\frac{\diam(X^{(k+1)}_{t(\om_{k+1})})}{\diam(X^{(k)}_{t(\om_k)})}\cdot\norm{D\phi_a^{(k+1)}}\cdot\frac{\norm{D\phi_{\om_{k+1}}^{(k+1)}}}{\norm{D\phi_a^{(k+1)}}}\label{insert into this}
	\end{align}
	for any $a\in L^{(k+1)}_{\om|_k}$. Estimating $\diam\left(\phi_{\om\rvert_{k}a}(X_{t(a)}^{(k)})\right)$ we see 
	\begin{align*}
	\diam(Y_{\om\rvert_k a})&\geq C^{-1}\norm{D\phi_{\om\rvert_k}}\diam\left(\phi_a^{(k+1)}(X_{t(a)}^{(k+1)})\right)\\
	&\geq C^{-2}\norm{D\phi_{\om\rvert_k}}\cdot\norm{D\phi_a^{(k+1)}}\diam(X_{t(a)}^{(k+1)})\\
	&\geq C^{-3}\frac{\diam(Y_{\om\rvert_k})}{\diam(X_{t(\om_k)}^{(k)})}\cdot\norm{D\phi_a^{(k+1)}}\cdot\diam(X_{t(a)}^{(k+1)}).
	\end{align*}
	Thus, solving for $\diam(Y_{\om_k})$, we see that
	\begin{align}\label{solving estimate}
	\diam(Y_{\om_k})\leq C^3\diam(Y_{\om\rvert_k a})\cdot\frac{\diam(X_{t(\om_k)}^{(k)})}{\diam(X_{t(a)}^{(k+1)})}\cdot\norm{D\phi_a^{(k+1)}}^{-1}.
	\end{align}
	Inserting \eqref{solving estimate} into \eqref{insert into this}, we have 
	\begin{align}
	\diam(Y_{\om|_{k+1}})&\leq C^2\diam(Y_{\om|_k})\cdot\norm{D\phi_a^{(k+1)}}\cdot\rho_{k+1}\cdot\frac{\diam\left(X^{(k+1)}_{t(\om_{k+1})}\right)}{\diam\left(X^{(k)}_{t(\om_k)}\right)}\nonumber\\
	&\leq C^5\diam(Y_{\om|_{k}a})\cdot\rho_{k+1}\cdot\frac{\diam\left(X^{(k+1)}_{t(\om_{k+1})}\right)}{\diam(X^{(k+1)}_{t(a)})}\nonumber\\
	&\leq C^5\cdot r\cdot \rho_{k+1}\cdot\frac{\diam\left(X^{(k+1)}_{t(\om_{k+1})}\right)}{\diam\left(X^{(k+1)}_{t(a)}\right)}.\label{diam Y om k+1}
	\end{align}
	Now estimating $\diam(Y_\om)$ we have 
	\begin{align*}
	\diam(Y_\om)&\leq C\norm{D\phi_{\om}}\cdot\diam\left(X^{(k+m)}_{t(\om)}\right)\\
	&\leq C\norm{D\phi_{\om|_{k+1}}}\cdot\norm{D\phi_{\om|_{k+2}^{k+m}}^{k+2,k+m}}\cdot\diam\left(X^{(k+m)}_{t(\om)}\right)\\
	&\leq C^2\frac{\diam(Y_{\om|_{k+1}})}{\diam(X^{(k+1)}_{t(\om_{k+1})})}\cdot \norm{D\phi_{\om|_{k+2}^{k+m}}^{k+2,k+m}}\cdot \diam(X^{(k+m)}_{t(\om)}).
	\end{align*}
	Inserting our estimate \eqref{diam Y om k+1} of $\diam(Y_{\om|_{k+1}})$, we now have 
	\begin{align*} 
	\diam(Y_\om)&\leq C^7\cdot r\cdot \rho_{k+1}\cdot\norm{D\phi_{\om|_{k+2}^{k+m}}^{k+2,k+m}}\cdot \frac{\diam\left(X^{(k+1)}_{t(\om_{k+1})}\right)}{\diam\left(X^{(k+1)}_{t(a)}\right)}\cdot\frac{\diam(X^{(k+m)}_{t(\om)})}{\diam(X^{(k+1)}_{t(\om_{k+1})})}\\
	&\leq C^7\cdot r\cdot \rho_{k+1}\cdot \eta^{m-1}\cdot\frac{\diam(X^{(k+m)}_{t(\om)})}{\diam(X^{(k+1)}_{t(a)})}\\
	&\leq C^7\cdot r\cdot \rho_{k+1}\cdot \eta^{m-1}\cdot\frac{\ol{d}_{k+m}}{\ul{d}_{k+1}}.
	\end{align*}
	Since for $\om\in W$ we have $\diam(Y_\om)\geq r$ we have that $r\leq rC^7\eta^{m-1}\rho_{k+1}\cdot\frac{\ol{d}_{k+m}}{\ul{d}_{k+1}}$ implies that
	\begin{align*}
	m&\leq\frac{\log\frac{C^7}{\eta}}{\log\eta^{-1}}\cdot\left[1+\log\rho_{k+1}+\log \frac{\ol{d}_{k+m}}{\ul{d}_{k+1}}\right]\\
	&= \const \left[1+\log\rho_{k+1}+\log\frac{\ol{d}_{k+m}}{\ul{d}_{k+1}} \right].
	\end{align*}
	Since $W_\ut(m)$ is a collection of words length $k+m$ of $W$, we must have that $\#W_\ut(m)\leq P$, hence letting 
	\begin{align*}
	W_\ut=\union_{m=0}^\infty W_\ut(m)=\union\set{W_\ut(m):0\leq m\leq \const\cdot\log\left[1+\rho_{\absval{\ut}+1}+\log\frac{\ol{d}_{\absval{\ut}+m}}{\ul{d}_{\absval{\ut}+1}}\right]},
	\end{align*}
	we see that 
	\begin{align*}
		\#W_\ut\leq\const\cdot\left[1+\log\rho_{\absval{\ut}+1}+\sup_{j\geq 1}\log\frac{\ol{d}_{\absval{\ut}+j}}{\ul{d}_{\absval{\ut}+1}}\right].
	\end{align*}
	
	Since $\#W$ is bounded, there is a constant $\Delta_B$ which depends on the ball $B$ with $\Delta_B=\max_{\om\in W}\absval{\om}<\infty$. Temporarily fixing $\om\in W$ with $\absval{\om}=n$, let 
	\begin{align*}
	C_\om=\set{a\in L_\om^{(n+1)}:Y_{\om a}\intersect B\nonempty \text{ and }\diam(Y_{\om a})\leq r}.
	\end{align*}
	Then 
	\begin{align*}
	B\intersect J\sub\union_{\om\in W}\union_{a\in C_\om}Y_{\om a}.
	\end{align*}
	Since the $Y_{\om a}$ have disjoint interiors for $\om\in W$ and $a\in C_\om$, and since $\union_{a\in C_\om}Y_{\om a}$ is contained in the disk of radius $2r$ with the same center as $B$, we have
	\begin{align*}
	\const\cdot r^d&\geq \sum_{a\in C_\om}\diam^d(Y_{\om a})\geq \sum_{a\in C_\om}C^{-d}\norm{D\phi_{\om a}}^d\diam^d(X_{t(a)}^{(n+1)})\\
	&\geq C^{-d}K^{-2d}\norm{D\phi_\om}^d\cdot(\ul{d}_{n+1})^d\cdot\sum_{a\in C_\om}\norm{D\phi_a^{(n+1)}}^d.
	\end{align*}
	From this, we get the estimate 
	\begin{align*}
	\norm{D\phi_\om}\leq\sqrt[d]{\frac{\const\cdot r^d}{(\ul{d}_{n+1})^d\sum_{a\in C_\om}\norm{D\phi_a^{(n+1)}}^d}}=\frac{\const\cdot r}{\ul{d}_{n+1}\sqrt[d]{\sum_{a\in C_\om}\norm{D\phi_a^{(n+1)}}^d}}.
	\end{align*}
	Now let $j\geq \Delta_B+p+1$. Using \eqref{mj of Yom} and \eqref{eqn:ineq} we can estimate the $j^{th}$ measure of $\union_{a\in C_\om}Y_{\om a}$ as
	\begin{align*}
	m_j\left(\union_{a\in C_\om}Y_{\om a}\right)&\leq \sum_{a\in C_\om}m_j(Y_{\om a})\leq \frac{\const\cdot e^{\ep (n+1)}}{Z_{n+1}(t)}\sum_{a\in C_\om}\norm{D\phi_{\om a}}^t\\
	&\leq \frac{\const\cdot e^{\ep (n+1)}\norm{D\phi_\om}^t}{Z_{n+1}(t)}\cdot\sum_{a\in C_\om}\norm{D\phi_a^{(n+1)}}^t\\
	&\leq \frac{\const\cdot e^{\ep (n+1)}}{Z_{n+1}(t)}\cdot\frac{\const\cdot r^t}{(\ul{d}_{n+1})^t\left(\sum_{a\in C_\om}\norm{D\phi_a^{(n+1)}}^d\right)^{\frac{t}{d}}}\cdot\sum_{a\in C_\om}\norm{D\phi_a^{(n+1)}}^t\\
	&=\frac{\const\cdot e^{\ep (n+1)}\cdot r^t}{(\ul{d}_{n+1})^t\cdot Z_{n+1}(t)}\cdot\frac{\sum_{a\in C_\om}\norm{D\phi_a^{(n+1)}}^t}{\left(\sum_{a\in C_\om}\norm{D\phi_a^{(n+1)}}^d\right)^{\frac{t}{d}}}.
	\end{align*}
	Setting $\vartheta:=\frac{d}{t}\geq 1$ and using H\"older's inequality we can estimate the fraction on the right to be
	\begin{align*}
	\frac{\sum_{a\in C_\om}\norm{D\phi_a^{(n+1)}}^t}{\left(\sum_{a\in C_\om}\norm{D\phi_a^{(n+1)}}^d\right)^{\frac{t}{d}}}&=\frac{\sum_{a\in C_\om}\norm{D\phi_a^{(n+1)}}^t}{\left(\sum_{a\in C_\om}\left(\norm{D\phi_a^{(n+1)}}^t\right)^\vartheta\right)^{\frac{1}{\vartheta}}}\\
	&\leq (\#C_\om)^{1-\frac{1}{\vartheta}}\leq \left(\#L_\om^{(n+1)}\right)^{1-\frac{t}{d}}\leq\left(\ol{G}_{n}\right)^{1-\frac{t}{d}}.
	\end{align*}
	Replacing our estimate into the inequality above we see that 
	\begin{align*}
	m_j\left(\union_{a\in C_\om}Y_{\om a}\right)
	&\leq\frac{\const\cdot r^t}{Z_{n+1}(t)}\cdot(\ul{d}_{n+1})^{-t}\cdot\left(\ol{G}_{n}\right)^{1-\frac{t}{d}}\cdot e^{\ep (n+1)}\\
	&=\frac{\const\cdot r^t}{Z_{n+1}(t)}\cdot (\ul{d}_{n+1})^{-t}\cdot\xi_{n}^{1-\frac{t}{d}}\cdot\left(\ul{G}_{n}\right)^{1-\frac{t}{d}}\cdot e^{\ep (n+1)}.
	\end{align*}
	Now, since $\xi_n\leq\ol{G_n}$, in light of our hypothesis \eqref{assumption 2}, we see that $\xi_{n}\leq e^{\dl' (n)}\leq e^{\dl' (n+1)}$ for sufficiently large $n$. This then gives us that $\xi_{n}^{1-\frac{t}{d}}\leq\const\cdot e^{\dl' (n+1)}$, and letting $\ep'=\dl'(p^2+p+1)$, we have 
	\begin{align}\label{eqn:mleqfrac}
	m_j\left(\union_{a\in C_\om}Y_{\om a}\right) 
	&\leq\frac{\const\cdot r^t}{Z_{n+1}(t)}\cdot(\ul{d}_{n+1})^{-t}\cdot\left(\ul{G}_{n}\right)^{1-\frac{t}{d}}\cdot\xi_{n}^{1-\frac{t}{d}}\cdot e^{\ep (n+1)}\notag\\
	&\leq \frac{\const\cdot r^t}{Z_{n+1}(t)}\cdot(\ul{d}_{n+1})^{-t}\cdot\left(\ul{G}_{n}\right)^{1-\frac{t}{d}}\cdot e^{(n+1)\cdot(\ep+\dl')}\notag\\
	&=\frac{\const\cdot r^t}{Z_{n+1}(t)}\cdot(\ul{d}_{n+1})^{-t}\cdot\left(\ul{G}_{n}\right)^{1-\frac{t}{d}}\cdot e^{\ep' (n+1)}.
	\end{align}

	Now we estimate $Z_{n+1}(t)$ 
	\begin{align}
	Z_{n+1}(t)&=\sum_{\gm\in I^{n+1}}\norm{D\phi_\gm}^t=\sum_{\beta\in I^n}\sum_{a\in L^{(n+1)}_\beta}\norm{D\phi_{\beta a}}^t\nonumber\\
	&\geq K^{-2t}\cdot\sum_{\beta\in I^n}\sum_{a\in L^{(n+1)}_\beta}\norm{D\phi_{\beta}}^t\norm{D\phi_a}^t
	\geq K^{-2t}\cdot\sum_{\beta\in I^n}\sum_{a\in L^{(n+1)}_\beta}\norm{D\phi_{\beta}}^t\cdot \ul{c}_{n+1}^t\nonumber\\
	&=K^{-2t}\cdot\ul{c}_{n+1}^t\sum_{\beta\in I^n}\sum_{a\in L^{(n+1)}_\beta}\norm{D\phi_{\beta}}^t
	=K^{-2t}\cdot\ul{c}_{n+1}^t\cdot \sum_{\beta\in I^n}\norm{D\phi_\beta}^t\cdot\#L^{(n+1)}_\beta\nonumber\\
	&\geq K^{-2t}\cdot\ul{c}_{n+1}^t\cdot\ul{G}_{n}\cdot\sum_{\beta\in I^n}\norm{D\phi_\beta}^t\nonumber\\
	&=K^{-2t}\cdot\ul{c}_{n+1}^t\cdot\ul{G}_{n}\cdot Z_n(t).\label{ineq *1 pg 20}
	\end{align}
	Thus replacing into our estimate of the measure of $\union_{a\in C_\om}Y_{\om a}$ we have 
	\begin{align*}
	m_j\left(\union_{a\in C_\om}Y_{\om a}\right)&\leq \frac{\const\cdot r^t\cdot e^{\ep' (n+1)}\left(\ul{G}_{n}\right)^{1-\frac{t}{d}}(\ul{d}_{n+1})^{-t}}{Z_{n+1}(t)}\\
	&\leq\frac{\const\cdot r^t\cdot e^{\ep' (n+1)}\cdot \left(\ul{G}_{n}\right)^{1-\frac{t}{d}}(\ul{d}_{n+1})^{-t} }{K^{-2t}\cdot Z_{n}(t)\cdot\ul{c}_{n+1}^t\cdot \ul{G}_{n}}\\
	&=\frac{\const\cdot r^t\cdot e^{\ep' (n+1)}}{Z_{n}(t)\cdot\ul{c}_{n+1}^t\cdot \left(\ul{G}_{n}\right)^{\frac{t}{d}}(\ul{d}_{n+1})^{t}}
	=\frac{\const\cdot e^{\ep' (n+1)}\cdot r^t}{\widetilde{Z}_{n+1}(t)}.
	\end{align*}	
	Since $\kp'<\kp(t)=\liminf_{n\to\infty}\frac{1}{n}\log\frac{\widetilde{Z}_n(t)}{1+\log\max_{j\leq n}\rho_j+\sup_{k\geq 1}\log\frac{\ol{d}_{n+k}}{\ul{d}_{n+1}} }$, then for sufficiently large $n$, we have 
	\begin{align*}
	e^{\kp'n}\leq\frac{\widetilde{Z}_n(t)}{1+\log\max_{j\leq n}\rho_j+\sup_{k\geq 1}\log\frac{\ol{d}_{n+k}}{\ul{d}_{n+1}}}.
	\end{align*}
	Now letting $\ut\in W'$ and $\om\in W_\ut$ and recalling that $\ep'=\dl'(p^2+p+1)<\kp'$ we have
	\begin{align*}
	m_j\left(\union_{a\in C_\om}Y_{\om a}\right)
	&\leq \frac{\const\cdot e^{\ep' (n+1)}}{\widetilde{Z}_{n+1}(t)}\cdot r^t\\
	&\leq \frac{\const\cdot e^{\ep' (n+1)}}{ e^{\kp'(n+1)}\cdot\left(1+\log\max_{j\leq n+1}\rho_j+\sup_{k\geq 1}\log\frac{\ol{d}_{n+k}}{\ul{d}_{n+1}}\right)}\cdot r^t\\
	&\leq\frac{\const}{1+\log\rho_{\absval{\ut}+1}+\sup_{k\geq 1}\log\frac{\ol{d}_{n+k}}{\ul{d}_{n+1}}}\cdot e^{(\ep'-\kp')(n+1)}\cdot r^t\\
	&\leq\frac{\const}{1+\log\rho_{\absval{\ut}+1}+\sup_{k\geq 1}\log\frac{\ol{d}_{n+k}}{\ul{d}_{n+1}}}\cdot r^t.
	\end{align*}
	Then for $r$ sufficiently small and $n=\absval{\om}$ sufficiently large and noting that the cardinalities of $W$, $W'$, and $W_\ut$ are all bounded, we see that 
	\begin{align*}
	m_j(B)&\leq m_j\left(\union_{\ut\in W'}\union_{\om\in W_\ut}\union_{a\in C_\om}Y_{\om a}\right)\leq \sum_{\ut\in W'}\sum_{\om\in W_\ut}m_j\left(\union_{a\in C_\om}Y_{\om a}\right)\\
	&\leq\sum_{\ut\in W'}r^t\cdot\sum_{\om\in W_\ut}\frac{\const}{1+\log\rho_{\absval{\ut}+1}+\sup_{k\geq 1}\log\frac{\ol{d}_{n+k}}{\ul{d}_{n+1}}}\\
	&\leq \const\cdot r^t\sum_{\ut\in W'}\frac{\#W_\ut}{1+\log\rho_{\absval{\ut}+1}+\sup_{k\geq 1}\log\frac{\ol{d}_{n+k}}{\ul{d}_{n+1}}}\\
	&\leq\const\cdot r^t.
	\end{align*}
	Taking $m$ to be an arbitrary weak limit of the sequence $(m_j)_{j=1}^\infty$ we have that 
	\begin{align*}
	m(B)\leq \const\cdot r^t,
	\end{align*}
	and therefore applying the mass distribution principle, it follows that $\HD(J_\Phi)\geq t$.
\end{proof}
To conclude this section we would like to present a simpler lower bound for a system for which Bowen's formula holds.
\begin{remark}
 Note that the formula presented in \eqref{eqn:pq1lessthanfraction} more generally gives us that 
	\begin{align*}
		Z_n(t)\geq K^{-nt}Z_{(1)}(t)\cdots Z_{(n)}(t).
	\end{align*}
	So in particular, given a NCGDMS $\Phi$ for which Bowen's formula holds, if there is some $M>1$ such that
	\begin{align*}
		Z_{(n)}(t)\geq MK^t
	\end{align*}
	for all $n\in\NN$ then $\ul{P}(t)>0$ and thus $\HD(J_\Phi)\geq t$. 	
\end{remark}
Even though this lower bound is less sophisticated than the one provided by Theorem \ref{thm:LBforHDJv}, we will be able to make good use of it later.

\section{Balancing Conditions}\label{sec: balancing}
We now present the conditions imposed upon the derivatives, which along with Theorem \ref{thm:LBforHDJv}, will ensure that Bowen's formula holds.
\begin{definition}
	A NCGDMS is called \textit{perfectly balanced} if $\rho_n=1$ for all $n\in\NN$ and furthermore $\phi^{(n)}_a$ are affine similarities. The system is called \textit{balanced} if there is a constant $C\geq 1$ such that $\rho_n\leq C$ for all $n\in\NN$ and \textit{weakly balanced} if 
	\begin{align*}
	\limit{n}{\infty}\frac{1}{n}\log\rho_n=0.
	\end{align*}
	The system is called \textit{barely balanced} if 
	\begin{align*}
	\limit{n}{\infty}\frac{\log(1+\log\rho_n)}{n}=0.
	\end{align*}
\end{definition}
Note that for the above definitions, we have the following chain of implications
\begin{align*}
\text{perfectly balanced }\Rightarrow \text{balanced }\Rightarrow \text{weakly balanced }\Rightarrow \text{barely balanced}.
\end{align*}
In the following lemma we will see that for barely balanced systems, the lower bound given in Theorem \ref{thm:LBforHDJv} depends only on the quantity
\begin{align*}
		\widetilde{Z}_{n}(t)&=Z_{n-1}(t)\cdot\left(\ul{G}_{n-1}\right)^{\frac{t}{d}}\cdot\ul{c}_{n}^t\cdot\ul{d}_n^t.
\end{align*}
\begin{lemma}\label{lem:LBbbs}
	Let $\Phi$ be a finitely primitive NCGDMS.
	\begin{enumerate}
		\item  The function
		\begin{align*}
		\widetilde{P}(t):=\liminf_{n\to\infty}\frac{1}{n}\log\widetilde{Z}_{n}(t)
		\end{align*}
		is strictly decreasing when it is finite and in fact 
		\begin{align*}
		\widetilde{P}(t')\geq \widetilde{P}(t)+(t-t')\log\frac{1}{\eta}
		\end{align*}
		when $t'<t$ and $\widetilde{P}(t)<\infty$. 
		\item Suppose that $\Phi$ is barely balanced, and that $t\geq 0$ such that
		\begin{align*}
		\widetilde{P}(t)\geq 0		
		\end{align*}	
		Then $\HD(J_\Phi)\geq t$.
	\end{enumerate}
	
\end{lemma}
\begin{proof}
	Noting that for each $a\in I^{(n)}$, $x\in X_{t(a)}^{(n)}$, and $r>0$ such that $B(x,r)\sub X_{t(a)}^{(n)}$ the Bounded Distortion Property gives us that 
	\begin{align*}
	X_{i(a)}^{(n-1)}\bus\phi_a^{(n)}(X_{t(a)}^{(n)})\bus\phi_a^{(n)}\left(B(x,r)\right)\bus B\left(\phi_a^{(n)}(x),K^{-1}\cdot r\cdot\norm{D\phi_a^{(n)}}\right).
	\end{align*}
	This implies that $\phi_a^{(n)}(X_{t(a)}^{(n)})$ must contain a ball of radius comparable to $\ul{c}_n$, and so the Open Set Condition must give us that there is some constant $L>0$, depending only on the dimension $d$, such that 
	\begin{align*}
		 \#I^{(n)}\ul{c}_n^d\leq L\cdot\#V_{n-1}\ol{d}_{n-1}^d.
	\end{align*}
	For each $n\in\NN$ define $C_n$ to be the number
	\begin{align*}
		C_n:=L\cdot\#V_{n-1}\cdot\ol{d}_{n-1}^d\cdot\ul{d}_n^d.
	\end{align*}	
	The Diameter Condition ensures that $C_n$ is subexponentially bounded. In particular, in view of Observation \ref{obs:AlphaIneq1}, we have that
	\begin{align}\label{eqn:CunderGoestoZero}
	\ul{G}_{n-1}\cdot \ul{c}_{n}^d\cdot\ul{d}_n^d\leq C_n.
	\end{align}
	Then for $t'<t$ and sufficiently large $n$
	\begin{align*}
	\frac{\widetilde{Z}_{n+1}(t)}{\widetilde{Z}_{n+1}(t')}&=\frac{Z_{n}(t)}{Z_{n}(t')}\cdot\left(\left(\ul{G}_{n}\right)^\frac{1}{d}\cdot\ul{c}_{n+1}\cdot\ul{d}_{n+1}\right)^{t-t'}
	\leq \frac{Z_{n}(t)}{Z_{n}(t')}\cdot C_n^{\frac{t-t'}{d}}\leq C_n^{\frac{t-t'}{d}}\cdot\eta^{n(t-t')}.
	\end{align*}
	Hence we arrive at the estimate
	\begin{align*}
	\widetilde{P}(t')=\liminf_{n\to\infty}\frac{1}{n+1}\log\widetilde{Z}_{n+1}(t')&\geq(t-t')\log\frac{1}{\eta}+\liminf_{n\to\infty}\left(\frac{\log Z_{n}(t)}{n}-\frac{(t-t')\log C_n}{nd}\right)\\
	&=(t-t')\cdot\log\frac{1}{\eta}+\widetilde{P}(t),
	\end{align*}
	which proves the first claim. 
	
	To prove the second claim, we begin by supposing that $\Phi$ is a barely balanced NCGDMS with $0\leq t'<t$ such that  $\widetilde{P}(t)\geq 0$. Then we have that $\widetilde{P}(t')> 0$ and thus $\widetilde{Z}_{n}(t')$ must grow at least exponentially in $n$. However, as $\Phi$ is barely balanced, the quantity 
	\begin{align*}
	1+\log\ds\max_{j\leq n}\set{\rho_j}+\sup_{k\geq 1}\log\frac{\ol{d}_{n+k}}{\ul{d}_{n+1}}
	\end{align*} 
	grows subexponentially due to the Diameter Condition. Hence $\kp(t')=\infty$. Applying Theorem \ref{thm:LBforHDJv}, we have that $\HD(J_\Phi)\geq t'$. As $t'<t$ was arbitrary we have that $\HD(J_\Phi)\geq t$. 	
\end{proof}

\begin{corollary}\label{cor:BFwbs}
	Let $\Phi$ be a weakly balanced and finitely primitive NCGDMS such that 
	\begin{align*}
	\limit{n}{\infty}\frac{1}{n}\log\ol{G}_n=0.
	\end{align*}
	Then $\HD(J_\Phi)=B_\Phi$.
\end{corollary}
\begin{proof}
	Notice that $\lim_{n\to\infty}\frac{1}{n}\log\ol{G}_n=0$ implies that 
	\begin{align*}
		\lim_{n\to\infty}\frac{1}{n}\log\ul{G}_n=0 \spand \lim_{n\to\infty}\frac{1}{n}\log\xi_n=0.
	\end{align*}
	Now let $t<B_\Phi$. Then $\ul{P}(t)>0$ and thus $Z_{n}(t)$ grows exponentially in $n$. Hence 
	\begin{align}
	Z_{n+1}(t)&=\sum_{\beta\in I^n}\sum_{a\in L^{(n+1)}_\beta}\norm{D\phi_{\beta a}}^t\leq \sum_{\beta\in I^n}\norm{D\phi_\beta}^t\sum_{a\in L^{(n+1)}_\beta}\norm{D\phi_a^{(n+1)}}^t\nonumber\\
	&\leq Z_{n}(t)\cdot \ol{G}_{n}\cdot\ol{c}_{n+1}^t\nonumber\\
	&= Z_{n}(t)\cdot \ul{G}_{n}\cdot\xi_{n}\cdot\rho_{n+1}^t\cdot\ul{c}_{n+1}^t\nonumber\\
	&=\widetilde{Z}_{n+1}(t)\cdot \left(\ul{G}_{n}\right)^{1-\frac{t}{d}}\cdot\xi_{n}\cdot\rho_{n+1}^t\cdot\ul{d}_{n+1}^{-t}.\label{eqn from cor 7.3}
	\end{align}
	Since the system is weakly balanced, the quantity  $\left(\ul{G}_{n}\right)^{1-\frac{t}{d}}\cdot\xi_{n}\cdot\rho_{n+1}^t\cdot\ul{d}_{n+1}^{-t}$ grows at most subexponentially in $n$, and so 
	\begin{align*}
	\widetilde{P}(t)=\liminf_{n\to\infty}\frac{1}{n}\log \widetilde{Z}_{n}(t)\geq\liminf_{n\to\infty}\frac{1}{n}\log Z_{n}(t)=\ul{P}(t)>0.
	\end{align*}
	Thus applying Lemma \ref{lem:LBbbs} implies that $\HD(J_\Phi)\geq t$. As $t< B_\Phi$ was chosen arbitrarily we have that $\HD(J_\Phi)\geq B_\Phi$, which in light of Lemma \ref{lem:HDlessB}, gives the desired result.
\end{proof}
Together with Corollary \ref{cor: fin prim + G implies subexp}, Corollary \ref{cor:BFwbs} completes the proof of Theorem \ref{MainThm}. Now we consider systems whose alphabets grow exponentially in size. 

\begin{corollary}\label{cor:BFforExpGrowth} Suppose that $\Phi$ is a finitely primitive NCGDMS such that 
	\begin{align*}
	\limsup_{n\to\infty}\frac{1}{n}\log\ol{G}_n<\infty \quad\text{ and }\quad \limsup_{n\to\infty}\frac{1}{n}\log\rho_n<\infty.
	\end{align*}
	If $\limit{n}{\infty}\ol{c}_n=0$, then $\HD(J_\Phi)=B_\Phi$.
\end{corollary}
\begin{proof}
	Fix $t<B_\Phi$. Then $Z_n(t)$ grows superexponentially which implies that $\widetilde{P}(t)>0$. 
	To see this let $\ep>0$ and let $t'=t-\ep$. Then for sufficiently large $n$ we have 
	\begin{align*}
	Z_n(t')&=\sum_{\om\in I^n}\norm{D\phi_\om}^{t-\ep}
	=\sum_{\om\in I^n}\norm{D\phi_\om}^t\norm{D\phi_\om}^{-\ep}
	\geq \sum_{\om\in I^n}\norm{D\phi_\om}^t\cdot\left(\max_{\om\in I^n}\norm{D\phi_\om}\right)^{-\ep}\\
	&= Z_n(t)\cdot\left(\max_{\om\in I^n}\norm{D\phi_\om}\right)^{-\ep}\geq Z_n(t)\cdot\prod_{j=1}^n(\ol{c}_j)^{-\ep}
	\geq \prod_{j=1}^n(\ol{c}_j)^{-\ep}.
	\end{align*}
	Since $\overline{c}_n\to 0$ then 
	\begin{align*}
	\frac{\sum_{i=1}^n-\log\ol{c}_n}{n}\to \infty,
	\end{align*}
	and thus we have that 
	\begin{align*}
	\frac{\log Z_n(t')}{n}\to\infty.
	\end{align*}
	In light of \eqref{eqn from cor 7.3}, we have 
	\begin{align*}
	\widetilde{Z}_n(t')\geq\frac{Z_n(t')}{\left(\ul{G}_{n-1}\right)^{1-\frac{t'}{d}}\cdot\rho_n^{t'}\cdot\xi_{n-1}\cdot\ul{d}_n^{-t'}}.
	\end{align*}
	As the numerator grows superexponentially and the denominator grows at most exponentially with $n$ by assumption, we see that $\widetilde{Z}_n(t')\to\infty$, and thus $\HD(J_\Phi)>t'=t-\ep$. Since $t<B_\Phi$ and $\ep>0$ were chosen arbitrarily we see that $\HD(J_\Phi)\geq B_\Phi$, finishing the proof. 
\end{proof}
The following corollary provides explicit estimates for the Hausdorff dimension for systems with at most exponential growth. In particular, we are able to fully calculate the dimension for systems that have regular exponential growth. 
\begin{corollary}\label{cor:BFforWBfrac}
	Suppose that $\Phi$ is a finitely primitive NCGDMS such that 
	\begin{align*}
	0<a_0:=\liminfty{n}\frac{1}{n}\log\ul{G}_n \leq \limsupty{n}\frac{1}{n}\log\ol{G}_n:=a_1<\infty
	\end{align*}
	and 
	\begin{align*}
	0<b_0:=\liminfty{n}\frac{1}{n}\log\left(\frac{1}{\ol{c}_n}\right)\leq \limsupty{n}\frac{1}{n}\log\left(\frac{1}{\ul{c}_n}\right):=b_1<\infty.
	\end{align*}
	Then 
	\begin{align*}
		\frac{a_0}{b_1}\leq\HD(J)\leq\frac{a_1}{b_0}.
	\end{align*}
	Moreover if the associated general limit of the above limits exists and we have that $a:=a_0=a_1$ and $b:=b_0=b_1$, then
	\begin{align*}
		\HD(J_\Phi)=a/b.
	\end{align*}
\end{corollary}
\begin{proof}
	As the second assumption implies that $\limty{n}\ol{c}_n=0$ and that $\limsupty{n}\frac{1}{n}\log\rho_n<\infty$, Corollary \ref{cor:BFforExpGrowth} assures us that Bowen's formula holds. Thus it suffices to appropriately estimate $B_\Phi$. Referencing \eqref{ineq *1 pg 20} from the proof of Theorem \ref{thm:LBforHDJv}, and combining with the inequality
	\begin{align*}
		Z_{n+1}(t)\leq Z_n(t)\cdot\ul{G}_{n}\cdot\xi_{n}\cdot\ol{c}_{n+1}^t,
	\end{align*} 
	we see that for any $t\geq 0$ we have
	\begin{align*}
	K^{-2t}\cdot Z_n(t)\cdot\ul{G}_{n}\cdot\ul{c}_{n+1}^t\leq Z_{n+1}(t)\leq  Z_n(t)\cdot\ul{G}_{n}\cdot\xi_{n}\cdot\ol{c}_{n+1}^t.
	\end{align*}
	Dividing both sides by $Z_n(t)$ gives 
	\begin{align*}
	\const\cdot\ul{G}_{n}\cdot\ul{c}_{n+1}^t\leq\frac{Z_{n+1}(t)}{Z_n(t)}\leq\ol{G}_{n}\cdot\ol{c}_{n+1}^t.
	\end{align*}
	Then by assumption if $t>\frac{a_1}{b_0}$, the right hand side goes to 0 as $n\to\infty$. For $t<\frac{a_0}{b_1}$ the left hand side goes to $\infty$. Thus we see that $Z_n(t)$ tends to $0$ for $t>\frac{a_1}{b_0}$ and to $\infty$ for $t<\frac{a_0}{b_1}$. And thus we have that 
	\begin{align*}
		\frac{a_1}{b_0}\leq B_\Phi\leq\frac{a_0}{b_1}.
	\end{align*}
\end{proof}
The next theorem concerns systems with extremal dimension.
\begin{theorem}\label{thm:ExtrHDandBC}
	Let $\Phi$ be a finitely primitive NCGDMS such that 
	\begin{align*}
	\limsup_{n\to\infty}\frac{\log\ol{G}_n}{n}<\infty \spand \limsup_{n\to\infty}\frac{\log\rho_n}{n}<\infty.
	\end{align*}
	If $\HD(J_\Phi)=0$ then $B_\Phi=0$ and if $B_\Phi=d$ then $\HD(J_\Phi)=d$.
\end{theorem}
\begin{proof}
	Recall that for any $t\geq 0$ we have the function 
	\begin{align*}
	\widetilde{P}(t)=\liminf_{n\to\infty}\frac{1}{n}\log \widetilde{Z}_n(t).
	\end{align*}
	Now, in light of \eqref{xi less than G up }, set 
	\begin{align*}
	\ul{a}=\liminf_{n\to\infty}\frac{\log\ul{G}_n}{n}\geq 0, \qquad \ol{a}=\limsup_{n\to\infty}\frac{\log\ol{G}_n}{n}<\infty, \qquad \ol{\xi}=\limsup_{n\to\infty}\frac{\log\xi_n}{n}.
	\end{align*}
	To see the first claim by way of contraposition, suppose that $B_\Phi>0$ and let 
	\begin{align*}
	\overline{\rho}=\limsup_{n\to\infty}\frac{\log\rho_n}{n}<\infty.
	\end{align*}
	Now fix $0<t_0<B_\Phi$, so that $\ul{P}(t_0)>0$, and recall that
	\begin{align*}
	Z_{n+1}(t_0)\leq Z_n(t_0)\cdot\ol{G}_{n}\cdot\ol{c}_{n+1}^{t_0}=Z_n(t_0)\cdot\ul{G}_{n}\cdot\xi_{n}\cdot\rho_{n+1}^{t_0}\cdot\ul{c}_{n+1}^{t_0}.
	\end{align*}
	Taking a logarithm, subtracting, and then dividing each side by $n$ leaves us with 
	\begin{align*}
	\frac{\log Z_{n+1}(t_0)-\log\ul{G}_{n}-\log\xi_{n}-t_0\log\rho_{n+1}}{n}\leq\frac{\log Z_n(t_0)+t_0\log\ul{c}_{n+1}}{n}.
	\end{align*}
	Taking a liminf of both sides, we see
	\begin{align}\label{eqn:ZnTildePressureIneq}
	-\infty<\ul{P}(t_0)-\ol{a}-t_0\ol{\rho}-\ol{\xi}\leq\liminf_{n\to\infty}\frac{\log Z_n(t_0)+t_0\log\ul{c}_{n+1}}{n}.
	\end{align}
	Now let $0<t<t_0$. Then we have 
	\begin{align*}
	\log\widetilde{Z}_{n+1}(t)&=\log Z_n(t)+\frac{t}{d}\log\ul{G}_{n}+t\log\ul{c}_{n+1}+t\log\ul{d}_{n+1}\\
	&>\log Z_n(t_0)+\frac{t}{d}\log\ul{G}_{n}+t\log\ul{c}_{n+1}+t\log\ul{d}_{n+1}\\
	&=\left(1-\frac{t}{t_0}\right)\log Z_n(t_0)+\frac{t}{d}\log\ul{G}_{n}+\frac{t}{t_0}\left(\log Z_n(t_0)+t_0\log\ul{c}_{n+1}+t_0\log\ul{d}_{n+1}\right).
	\end{align*}
	In light of \eqref{eqn:ZnTildePressureIneq} and the Diameter Condition, we have 
	\begin{align*}
	\widetilde{P}(t)\geq \left(1-\frac{t}{t_0}\right)\cdot\ul{P}(t_0)+\frac{t}{d}\cdot\ul{a}+\frac{t}{t_0}\cdot\left(\ul{P}(t_0)-\ol{a}-t_0\ol{\rho}-\ol{\xi}\right).
	\end{align*}
	Letting $t\to 0$, and recalling that $B_\Phi>0$, we see that the right hand side tends to $\ul{P}(t_0)$, but for sufficiently small $t>0$, we have 
	$\widetilde{P}(t)>0$. Hence, Lemma \ref{lem:LBbbs} gives us that $\HD(J_\Phi)\geq t>0$.
	
	For the second claim suppose that $B_\Phi=d$. Let $t'<d$ and $t\in(t',d)$ to be chosen sufficiently close to $d$ later. We recall the proof of Theorem \ref{thm:LBforHDJv}. Let $m$ be the measure constructed there and let $B$ be a ball of radius $r>0$. We continue the proof analogously up to \eqref{eqn:mleqfrac}, which says that
	\begin{align*}
	m\left(\union_{a\in C_\om}Y_{\om a}\right)\leq\frac{\const\cdot r^t}{Z_{n+1}(t)}\cdot\left(\ul{G}_{n}\right)^{1-\frac{t}{d}}\cdot e^{\ep' (n+1)}\cdot\ul{d}_{n+1}^{-t}.
	\end{align*}
	Recall that $\ul{G}_{n}$ grows at most exponentially fast and that we have $n=\absval{\om}$, where $\om$ was such that 
	\begin{align*}
	r\leq\diam(Y_\om)\leq\const\cdot\norm{D\phi_\om}\leq\const\cdot\eta^n.
	\end{align*}
	Thus we have
	\begin{align*}
	m\left(\union_{a\in C_\om}Y_{\om a}\right)&\leq\frac{\const\cdot r^{t'}}{Z_{n+1}(t)}\cdot r^{t-t'}\cdot\left(\ul{G}_{n}\right)^{1-\frac{t}{d}}\cdot e^{\ep' (n+1)}\cdot\ul{d}_{n+1}^{-t}\\
	&\leq \frac{\const\cdot r^{t'}}{Z_{n+1}(t)}\cdot\eta^{n(t-t')}\cdot\left(\ul{G}_{n}\right)^{1-\frac{t}{d}}\cdot e^{\ep'(n+1)}\cdot\ul{d}_{n+1}^{-t}.
	\end{align*}
	If $t$ is chosen sufficiently close to $d$ (depending on $\eta$, the exponential growth of $\ul{G}_{n}$, and $t'$), then $\eta^{n(t-t')}$ tends to 0 faster than $\left(\ul{G}_{n+1}\right)^{1-\frac{t}{d}}$ tends to $\infty$. Furthermore, the Diameter Condition implies that for sufficiently large $n$, we have $\ul{d}_{n+1}\leq e^{\ep'(n+1)}$. Hence, for large $n$, we have 
	\begin{align}
	m\left(\union_{a\in C_\om}Y_{\om a}\right)&\leq\frac{\const\cdot r^{t'}}{Z_{n+1}(t)}\cdot e^{(1-t)\ep'(n+1)}.\label{insert zn grows exp here}
	\end{align}
	Now let $\ut\in W'$ and $\om\in W(\ut)$. Since $\rho_n$ grows at most exponentially, and $Z_n(t)$ grows at least exponentially, we see that for sufficiently large $n$
	\begin{align}
		Z_{n+1}(t)\geq e^{(1-t)\ep'(n+1)}\left(1+\log\rho_{\absval{\ut}+1}+\sup_{k\geq 1}\log\frac{\ol{d}_{n+k}}{\ul{d}_{n+1}}\right)\label{zn grows exp rho not}.
	\end{align}
	Now inserting \eqref{zn grows exp rho not} into \eqref{insert zn grows exp here}, we see
	\begin{align*}
	m\left(\union_{a\in C_\om}Y_{\om a}\right)&\leq\frac{\const\cdot r^{t'}}{Z_{n+1}(t)}\cdot e^{(1-t)\ep'(n+1)}
	\leq \frac{\const\cdot r^{t'}\cdot e^{(1-t)\ep'(n+1)}}{e^{\ep'(n+1)}\left(1+\log\rho_{\absval{\ut}+1}+\sup_{k\geq 1}\log\frac{\ol{d}_{n+k}}{\ul{d}_{n+1}}\right)}\\
	&=\frac{\const\cdot r^{t'}}{1+\log\rho_{\absval{\ut}+1}+\sup_{k\geq 1}\log\frac{\ol{d}_{n+k}}{\ul{d}_{n+1}}}.
	\end{align*}
	Continuing as in the proof of Theorem \ref{thm:LBforHDJv}, we see that $\HD(J_\Phi)\geq t'$. As $t'<d$ was arbitrary we have that $\HD(J_\Phi)= B_\Phi=d$ as claimed.
\end{proof}

\section{Finite and Infinite Non-Stationary NCIFS}\label{section Rem-Urb NCIFS}

In \cite{rempe-gillen_non-autonomous_2016}, Rempe-Gillen and Urba\'nski were able to remove the balancing conditions and show that Bowen's formula holds for all subexponentially bounded and stationary NCIFS as well as for certain reasonable classes of infinite stationary NCIFS. In this short section, we generalize their results to the case of \textit{non-stationary} NCIFS. However, we do not include proofs as the corresponding proofs of their analogous stationary results from \cite{rempe-gillen_non-autonomous_2016} go through with perhaps only minor notational changes, if any. 

\begin{theorem}\label{thm: BF nonstat NCIFS}
	If $\Phi$ is a finite, non-stationary NCIFS which is subexponentially bounded, then Bowen's formula holds.  
\end{theorem}
As mentioned, the proof of this theorem is the same as Theorem 5.3 of \cite{rempe-gillen_non-autonomous_2016}. Rempe-Gillen and Urba\'nski accomplished this by showing that any subexponentially bounded stationary NCIFS has a subsystem, that is weakly balanced, for which the pressure is arbitrarily close to the pressure of the whole system. As Bowen's formula holds, for the weakly balanced subsystem, they were able to show that Bowen's formula must also hold for the whole system. The change from stationary to non-stationary does not affect the existence of such a subsystem, however it is worth pointing out that this same technique becomes much more difficult when considering a NCGDMS in its full generality. In particular, without further assumptions on which strings of letters are admissible, we can not form a satisfactory subsystem for which the pressure remains close to that of the original. In the sequel we will give two classes of systems with such additional assumptions for which we are able to remove the balancing conditions by forming appropriate subsystem.

We now introduce the notion of an infinite non-stationary NCIFS. As is the case for stationary systems, we will assume that each of the alphabets $I^{(n)}$ are infinite, and in fact equal to $\NN$ for convenience.
\begin{definition}\label{def: class M}
	Suppose that $\Phi$ is an infinite NCIFS, not necessarily stationary, and that $I^{(n)}=\NN$ for each $n\in\NN$. We say that$\Phi$ belongs to \textit{class $\cM$} if the following conditions are satisfied for each $t\in (0,d)$ and each $\ep>0$. 
	\begin{enumerate}
		\item The sums $Z_{(n)}(t)$ are either infinite or finite for all $n\in\NN$.
		\item If $Z_{(1)}(t)<\infty$, then 
		\begin{align*}
		\limty{n}\frac{\sum_{k\leq e^{\ep n}}\norm{D\phi_k^{(n)}}^t}{Z_{(n)}(t)}=1.
		\end{align*}
		\item If $Z_{(1)}(t)=\infty$, then
		\begin{align*}
		\limty{n}\sum_{k\leq e^{\ep n}}\norm{D\phi_k^{(n)}}^t=\infty.
		\end{align*}
	\end{enumerate}
\end{definition}
\begin{theorem}\label{thm: BF nonstat class M}
	If $\Phi$ is in class $\cM$, then Bowen's formula holds. 
\end{theorem}
Corollary 8.3 of \cite{rempe-gillen_non-autonomous_2016} shows that Bowen's formula holds for all stationary systems belonging to class $\cM$. This is accomplished by again showing that each infinite system in $\cM$ has a finite subsystem which is subexponentially bounded and weakly balanced whose pressure is arbitrarily close to that of the original system. However, as the above conditions are often quite difficult to check, the following definition was introduced. 
\begin{definition}\label{def evenly varying}
	An infinite NCIFS is said to be \textit{evenly varying} if there is a sequence $\seq[i]{\eta_i}$ of positive real numbers and a constant $c\geq 1$ such that 
	\begin{align*}
	\frac{\eta_i}{c}\leq\norm{D\phi_i^{(n)}}\leq c\cdot\eta_i.
	\end{align*}
\end{definition}
\begin{theorem}\label{thm: BF nonstat EV}
	If a NCIFS $\Phi$ is evenly varying, then it is in class $\cM$.
\end{theorem}
Proposition 8.5 of \cite{rempe-gillen_non-autonomous_2016} shows that if $\Phi$ is evenly varying, then it belongs to class $\cM$, hence Bowen's formula must hold. 

Again, we point out that Theorems \ref{thm: BF nonstat NCIFS}, \ref{thm: BF nonstat class M}, and \ref{thm: BF nonstat EV} only apply for NCIFS and not a general NCGDMS. Additionally, the proof techniques cannot be applied to NCGDMS as the existence of a suitable subsystem is not as intuitive as in the case of NCIFS. However, in Section \ref{Ascending} we introduce the class of ascending systems, which may be finite or infinite, for which we require no conditions on the derivatives nor conditions on the size of the alphabets.

\section{Hausdorff Measures}\label{Sec: Haus meas}
In this short section we present a result concerning the Hausdorff measure of the limit set for uniformly finite balanced systems. First we shall prove the following lemma.
\begin{lemma}\label{lem:UpperBoundHDforHMeasures}
	Let $\Phi$ be a NCGDMS. For each $n\in\NN$ let $\scr{U}_n$ be a covering of the limit set of the system $\Psi_n$ defined by $\Psi_n^{(j)}:=\Phi^{(n+j-1)}$ on the same sequence of spaces $\cX^{(j)}_{\Psi_n}=\cX_\Phi^{(n+j-1)}$ with the same sequence of incidence matrices $A_{\Psi_n}^{(j)}=A_\Phi^{(n+j-1)}$. Denote its Hausdorff sum by 
	\begin{align*}
		S( \scr{U}_n,t):=\sum_{U\in \scr{U}_n}\diam^t(U).
	\end{align*}
	Suppose that $t\geq 0$ is such that
	\begin{align*}
		\liminf_{n\to\infty}Z_{n-1}(t)\cdot S(\scr{U}_n,t)<\infty.
	\end{align*}
	Then $\HD(J_\Phi)\leq t$.
\end{lemma}
\begin{proof}
First note that $J_{\Psi_n}\sub\bigsqcup_{v\in V_{n-1}}X_v^{(n-1)}$. Now consider the collection
\begin{align*}
	\scr{V}_n=\set{\phi_\om(U):\om\in I^{n-1}, U\in\scr{U}_n}.	
\end{align*}
Then $\scr{V}_n$ is a covering of $J_\Phi$ and the diameters of the elements of $\scr{V}_n$ go to zero as $n$ tends toward infinity by the uniform contraction principle. 
Now each element of $\scr{V}_n$ is such that 
\begin{align*}
	\diam(\phi_\om(U))\leq C\cdot\norm{D\phi_\om}\cdot\diam(U),
\end{align*}
Thus we have 
\begin{align*}
	S(\scr{V}_n,t)&\leq C\cdot\left(\sum_{\om\in I^{n-1}}\norm{D\phi_\om}^t\right)\cdot\left(\sum_{U\in\scr{U}_n}\diam^t(U)\right)\\
	&=C\cdot Z_{n-1}(t)\cdot S(\scr{U}_n,t).
\end{align*} 
As we have $\liminf_{n\to\infty}Z_{n-1}(t)\cdot S(\scr{U}_n,t)<\infty$, we see that $H_t(J_\Phi)<\infty$, which implies that $\HD(J_\Phi)\leq t$, as claimed. 
\end{proof}
Recall that the function $\kp(t)$ is given by 
\begin{align*}
\kp(t)=\liminf_{n\to\infty} \frac{1}{n}\log\frac{\widetilde{Z}_{n}(t)}{1+\log\max_{j\leq n+1}\set{\rho_j}+\sup_{k\geq 0}\log(\ol{d}_{n+k}/\ul{d}_n)}, 
\end{align*}.
\begin{theorem}
	Suppose $\Phi$ is a balanced, uniformly finite, and $p$-finitely primitive NCGDMS such that there is some constant $0<\dl<\infty$ such that 
	\begin{align*}
		\dl^{-1}\leq \ul{d}_n\leq\ol{d}_n\leq \dl
	\end{align*}
	for each $n\in\NN$, and let $h=\HD(J_\Phi)=B_\Phi$. Then the $h$-dimensional Hausdorff measure $H_h(J_\Phi)$ is infinite, finite, or zero depending on whether 
	\begin{align*}
		\liminf_{n\to\infty}Z_n(h)
	\end{align*}
	is infinite, finite, or zero respectively. 	
\end{theorem}
\begin{proof}
Taking \eqref{ineq: butter knife} from the proof of Lemma \ref{lem:HDlessB} within the context of our current system, we see that
\begin{align*}
	H_h(J(\Phi))\leq \liminf_{n\to\infty}(C\cdot\ol{d}_n)^hZ_{n}(h)\leq (C\cdot\dl)^h\liminfty{n} Z_n(h).
\end{align*}
Thus the case in which the liminf is equal to zero is clear, leaving us to prove the claim in the case where the liminf is finite or infinite.  

Now from the proof of Theorem \ref{thm:LBforHDJv}, for any ball $B$ of radius $r>0$ we have 
\begin{align*}
	m(B)\leq\const\cdot r^h
\end{align*}
under the assumption that 
\begin{align*}
	0\leq (p^2+p+1)\cdot\limsup_{n\to\infty}\frac{1}{n}\log\ol{G}_n<\kp(h).
\end{align*} 
The mass distribution principle then implies that $H_h(J(\Phi))$ is positive or infinite if $\kp(h)$ is positive or infinite respectively. 

Since the system is balanced we have that $\rho_n$ is uniformly bounded, and since the system is uniformly finite we have that both $\ol{G}_n$ and $\#I^{(n)}$ are both uniformly bounded. Furthermore, by assumption we have that the diameters of the spaces $X_v^{(n)}$ are uniformly bounded away from 0 and $\infty$. Thus, we have that 
\begin{align*}
	\widetilde{Z}_n(h)&=Z_n(h)\cdot\left(\ul{G}_{n+1}\right)^{\frac{h}{d}}\cdot\ul{c}_{n+1}^{h}\cdot\ul{d}_{n+1}^{h}\leq \const\cdot Z_n(h).
\end{align*} 
 Hence, $\kp(h)$ is comparable to $\liminfty{n} Z_n(h)$ which proves the claim. 
\end{proof}

\section{Relaxing Subexponential Boundedness and Finite Primitivity}\label{Sec: Relax Subexp}

In this section we show that Bowen's formula holds for two classes of NCGDMS for which we are able to relax or remove either of the hypotheses of finite primitivity and subexponential boundedness. 

\begin{theorem}\label{thm: BF for g bdd}
	Let $S$ be a finite NCGDMS that is $p$-finitely primitive. Additionally we assume that there is some constant $M>0$ such that for each $n\in\NN$ we have 
	\begin{align*}
		Z_{(n)}(t)=\sum_{i\in I^{(n)}}\norm{D\phi_i}^t\leq M.
	\end{align*}
	We also suppose that there is a function $g:\NN\to\NN$ such that the following hold for each $\ell\in\NN$:
	\begin{enumerate}
		\item $\#I^{(k(\ell+p)+1)}\leq g(\ell)$ for each $k\in\NN$,
		\item $\#I^{(k(\ell+p)-p)}\leq g(\ell)$ for each $k\in\NN$,
		\item $\lim_{\ell\to\infty}\frac{1}{\ell}\log g(\ell)=0$.
	\end{enumerate}
	Then Bowen's formula holds.
\end{theorem}
\begin{remark}\label{rem: thm g bdd applies for subexp}
	Although the above conditions concerning derivatives may seem restrictive, it should be noted that this condition is met by all uniformly bounded alphabets, which means that this theorem generalizes all finite autonomous IFS and GDMS. Moreover, clearly we have that if $\Phi$ is subexponentially bounded then such a function $g$ exists.
\end{remark}

\begin{proof}

Let $\ell\in\NN$ and for any $m\in\NN$, $a\in I^{(m)}$, and $b\in I^{(m+\ell)}$ we define the partition function 
\begin{align*}
	Z_{m,m+\ell}(t;(a,b)):=\sum_{\mathclap{\substack{\om\in I^{m,m+\ell}\\ \om|_m=a\text{, }\om|_{m+\ell}=b}}}\norm{D\phi_\om}^t.
\end{align*}
Then clearly, we have 
\begin{align*}
	Z_{m,m+\ell}(t)=\sum_{\mathclap{\substack{a\in I^{(m)}\\ b\in I^{(m+\ell)}}}}Z_{m,m+\ell}(t;(a,b)).
\end{align*}

For each $\ell>1$ define the following alphabets
\begin{align*}
	B^{(1)}=B_\ell^{(1)}&=:I^{(1)}\times I^{(\ell)},\\
	B^{(2)}=B_\ell^{(2)}&=:I^{(\ell+p+1)}\times I^{(2\ell+p)},\\
	B^{(3)}=B_\ell^{(3)}&=:I^{(2\ell+2p+1)}\times I^{(3\ell+2p)},\\
	B^{(n)}=B_\ell^{(n)}&=:I^{([n-1]\ell+[n-1]p+1)}\times I^{(n\ell+[n-1]p)}.			
\end{align*} 

For $1\leq n\leq \infty$ we define $B^n=\prod_{j=1}^n B^{(j)}$. For finite $n$ we denote an element of $B^n$ by the pair $(a,b):=((a_1,b_1),(a_2,b_2),\dots,(a_n,b_n))$. We use similar notation to denote an element of $B^\infty$.
For each $n\in\NN$ we let $(a^*_n,b^*_n)\in B^{(n)}$ be the pair which maximizes the partition function $Z_{n,n+\ell}(t;(a,b))$. Notice that by finite primitivity for each $n\in\NN$ there is some $\lm^{(n)}\in\Lambda_{n\ell+(n-1)p}$ such that $b^*_n\lm^{(n)}a^*_{n+1}$ is admissible. 

Then for each $\ell\in\NN$ we can form a NCIFS
\begin{align*}
	S_\ell=\set{\phi_e^{(n)}: X_{t(e)}^{n(\ell+p)}\to X_{i(e)}^{(n-1)(\ell+p)+1}:n\in\NN, e\in E_\ell^{(n)}},
\end{align*} where the alphabets $E_\ell^{(j)}$ are given by 
\begin{align*}
	E_\ell^{(1)}&=\set{\gm\in I^{\ell+p}:\gm=\om\lm^{(1)},\om_1=a^*_1, \om_\ell=b^*_1},\\
	E_\ell^{(2)}&=\set{\gm\in I^{\ell+p+1,2(\ell+p)}: \gm=\om\lm^{(2)},\om_{\ell+p+1}=a^*_2, \om_{2\ell+p}=b^*_2},\\
	E_\ell^{(3)}&=\set{\gm\in I^{2(\ell+p)+1,3(\ell+p)}: \gm=\om\lm^{(3)},\om_{2(\ell+p)+1}=a^*_3, \om_{3(\ell+p)}=b^*_3},\\
	E_\ell^{(n)}&=\set{\gm\in I^{(n-1)(\ell+p)+1,n(\ell+p)}: \gm=\om\lm^{(n)},\om_{(n-1)(\ell+p)+1}=a^*_n, \om_{n(\ell+p)}=b^*_n}.
\end{align*}
In fact, $S_\ell$ is a subexponentially bounded NCIFS. To see this we note that since $S$ is subexponentially bounded we have 
\begin{align*}
	\#E_\ell^{(n)}&\leq \#I^{(n-1)(\ell+p)+1,n(\ell+p)}
	\leq \#I^{([n-1][\ell+p]+1)}\cdots\#I^{(n[\ell+p])}\\
	&\leq e^{\ep[(n-1)(\ell+p)+1]}\cdots e^{\ep[n[\ell+p]]}\\
	&=e^{\frac{\ep}{2}[(\ell+p)^2(2n-1)+\ell+p]}.
\end{align*}
Thus we have that 
\begin{align*}
	\frac{1}{n}\log\#E^{(n)}&\leq \frac{\ep((\ell+p)^2(2n-1)+\ell+p)}{2n}.
\end{align*}
Hence we see that for fixed $\ell$ we have
\begin{align*}
	\lim_{n\to\infty}\frac{1}{n}\log\#E^{(n)}\leq \ep(\ell+p)^2
\end{align*}
for every $\ep>0$. So, letting $\ep\to 0$ gives that $S_\ell$ is subexponentially bounded.

Now, for ease of notation, we define the following additional alphabets for each $\ell>1$ and $n>0$: 
\begin{align*}
	\Omega^{(1)}=\Omega_\ell^{(1)}&:=I^\ell, 				&	T^{(1)}=T_\ell^{(1)}&:=I^{\ell+1,\ell+p},			\\
	\Omega^{(2)}=\Omega_\ell^{(2)}&:=I^{\ell+p+1,2\ell+p},	&	T^{(2)}=T_\ell^{(2)}&:=I^{2\ell+p+1,2\ell+2p},		\\
	\Omega^{(3)}=\Omega_\ell^{(3)}&:=I^{2\ell+2p+1,3\ell+2p},&	T^{(3)}=T_\ell^{(3)}&:=I^{3\ell+2p+1,3\ell+3p},		\\
	\Omega^{(n)}=\Omega_\ell^{(n)}&:=I^{(n-1)\ell+(n-1)p+1,n\ell+(n-1)p},&	T^{(n)}=T_\ell^{(n)}&:=I^{n\ell+(n-1)p+1,n\ell+np}	.
\end{align*}
Then words in $\Om^{(j)}$ are of length $\ell$ and words in $T^{(j)}$ are of length $p$ for each $j$.
Using this notation we get 
\begin{align*}
	\sum_{\om\in \Omega^{(n)}}\norm{D\phi_\om}^t&=\sum_{(a,b)\in B^{(n)}}
	\,\,
	\sum_{\mathclap{\substack{\om\in\Omega^{(n)}\\ \om|_1=a\text{, }\om|_\ell=b}}}\norm{D\phi_\om}^t
	\leq \#B^{(n)}
	\cdot\sum_{\mathclap{\substack{\om\in\Omega^{(n)}\\ \om|_1=a^*_n\text{, }\om|_\ell=b^*_n}}}\norm{D\phi_\om}^t.
\end{align*}
And thus we can estimate the partition function 
\begin{align*}
	Z_{n(\ell+p)}(t)&\leq 
	\left(\sum_{\om\in \Omega^{(1)}}\norm{D\phi_\om}^t\right)\cdot\left(\sum_{\ut\in T^{(1)}}\norm{D\phi_\ut}^t\right)
	\dots\left(\sum_{\om\in \Omega^{(n)}}\norm{D\phi_\om}^t\right)\cdot
	\left(\sum_{\ut\in T^{(n)}}\norm{D\phi_\ut}^t\right)\\
	&\leq \prod_{j=1}^n\#B^{(j)}\sum_{\mathclap{\substack{\om\in\Omega^{(1)}\\ \om|_1=a^*_1\text{, }\om|_\ell=b^*_1}}}\norm{D\phi_\om}^t
	\left(\sum_{\ut\in T^{(1)}}\norm{D\phi_\ut}^t\right)\cdots
	\sum_{\mathclap{\substack{\om\in\Omega^{(n)}\\ \om|_1=a^*_n\text{, }\om|_\ell=b^*_n}}}\norm{D\phi_\om}^t
	\left(\sum_{\ut\in T^{(n)}}\norm{D\phi_\ut}^t\right)\\
	&\leq M^{np}\prod_{j=1}^n\#B^{(j)}\sum_{\mathclap{\substack{\om\in\Omega^{(1)}\\ \om|_1=a^*_1\text{, }\om|_\ell=b^*_1}}}\norm{D\phi_\om}^t\cdots
	\sum_{\mathclap{\substack{\om\in\Omega^{(n)}\\ \om|_1=a^*_n\text{, }\om|_\ell=b^*_n}}}\norm{D\phi_\om}^t\\
	&\leq \left(\frac{M^p}{Q^t}\right)^n\prod_{j=1}^n\#B^{(j)}\sum_{\mathclap{\substack{\om\in\Omega^{(1)}\\ \om|_1=a^*_1\text{, }\om|_\ell=b^*_1}}}\norm{D\phi_\om}^t\cdot\norm{D\phi_{\lm^{(1)}}}^t\cdots
	\sum_{\mathclap{\substack{\om\in\Omega^{(n)}\\ \om|_1=a^*_n\text{, }\om|_\ell=b^*_n}}}\norm{D\phi_\om}^t\cdot\norm{D\phi_{\lm^{(n)}}}^t\\
	&\leq\left(\frac{K^2M^p}{Q^t}\right)^n\prod_{j=1}^n\#B^{(j)}Z^{S_\ell}_n(t).
\end{align*}
Due to the properties of the bounding function $g(\ell)$ we can estimate the product in the last line above as
\begin{align*}	
	\prod_{j=1}^n\#B^{(j)}=\prod_{j=1}^n \#I^{([j-1]\ell+[j-1]p+1)}\cdot\# I^{(j\ell+[j-1]p)}
	\leq\prod_{j=1}^{n} g^2(\ell)
	=g^{2n}(\ell).
\end{align*}
Thus we see that 
\begin{align*}
	\frac{1}{n(\ell+p)}\log\prod_{j=1}^n\#B^{(j)}\leq \frac{1}{n(\ell+p)}\log g^{2n}(\ell)=\frac{2\log g(\ell)}{\ell+p}.
\end{align*}
Now note that we have 
\begin{align*}
	Z_{n(\ell+p)}(t)\geq Z_n^{S_\ell}(t),
\end{align*}
which implies that 
\begin{align*}
	\ul{P}(t)\geq\frac{1}{\ell+p}\ul{P}^{S_\ell}.
\end{align*}
Combining the above inequalities we are able to obtain the following estimate for the pressure: 
\begin{align*}
	\frac{1}{\ell+p}\ul{P}^{S_\ell}\leq\ul{P}(t)\leq \frac{1}{\ell+p}\log\left(\frac{K^2M^p}{Q^t}\right)+\frac{2\log g(\ell)}{\ell+p}+\frac{1}{\ell+p}\ul{P}^{S_\ell}(t).
\end{align*}
Clearly $B_{S_\ell}\leq B_S$ from the inequality on the left. Now if we let $t<B_S$, which implies that $\ul{P}(t)>0$. Thus for sufficiently large $\ell$, we must have that $\ul{P}^{S_\ell}(t)>0$, in fact, for fixed $t$, $\ul{P}^{S_\ell}(t)\to\infty$ as $\ell\to\infty$. Hence $t<B_{S_\ell}\leq B_S$, and so we have 
\begin{align*}
	\limit{\ell}{\infty}B_{S_\ell}=B_S.
\end{align*} 
As $S_\ell$ is a subexponentially bounded NCIFS, Theorem \ref{thm: BF nonstat NCIFS} gives that Bowen's formula holds, and thus, $\HD(J_{S_\ell})=B_{S_\ell}$. But since $J_{S_\ell}\sub J_S$ we have that $\HD(J_S)\geq\HD(J_{S_\ell})$. Thus we have that $\HD(J_S)\geq \HD(J_{S_\ell})=B_{S_\ell}$ for all $\ell$, and hence we have $\HD(J_S)\geq B_S$.

In light of the fact that we always have $\HD(J_S)\leq B_S$, we see that Bowen's formula holds for the system $S$, which finishes the proof.   
\end{proof}
\begin{remark}
	It is worth mentioning that we may weaken the hypothesis on $Z_{(n)}(t)$ in the previous theorem. In fact it suffices to know that there is a sequence $\seq[j]{M_j}$ of positive integers with
	\begin{align*}
		Z_{j+1,j+p}(t)\leq M_j
	\end{align*}
	such that
	\begin{align*}
		\limty{n}\frac{\sum_{k=1}^n\log M_{kj+(k-1)p}}{n}=0
	\end{align*}
	for each $j\in\NN$.
\end{remark}
The following corollary follows immediately from Theorem \ref{thm: BF for g bdd} in view of Remark \ref{rem: thm g bdd applies for subexp}.
\begin{corollary}
	Suppose $S$ is a finite NCGDMS that is $p$-finitely primitive and subexponentially bounded such that there is some constant $M>0$ such that for each $n\in\NN$ we have 
	\begin{align*}
	Z_{(n)}(t)=\sum_{i\in I^{(n)}}\norm{D\phi_i}^t\leq M.
	\end{align*}
	Then Bowen's formula holds.
\end{corollary}

In the following theorem we are able to remove the balancing conditions from Section \ref{sec: balancing} by again appealing to Theorem \ref{thm: BF nonstat NCIFS}. We will accomplish this using a NCGDMS which is ``pinched" to a single vertex sufficiently often as to allow us to reimagine the given NCGDMS as a subexponentially bounded non-stationary NCIFS. It is precisely this pinching condition that allows us to remove the requirement of finite primitivity as we will consider the words in between the pinched points to be the letters of the generated NCIFS. 
\begin{theorem}
	Suppose that $\Phi$ is a finite subexponentially bounded NCGDMS and that there is a strictly increasing sequence $\set{\ell_j}_{j\geq 1}$ of natural numbers such that the following hold:
	\begin{enumerate}
		\item There is some constant $M>0$ such that
		\begin{align*}
		\limsup_{n\to\infty}\frac{\ell_n^2-\ell_{n-1}^2}{n}\leq M.
		\end{align*}
		\item For each $j\in\NN$ the set $V_{\ell_j}$ is a singleton and the incidence matrix $A^{(\ell_j)}_{ab}=1$ for all $a\in I^{(\ell_j)}$ and $b\in I^{(\ell_j+1)}$, in other words, that the matrices $A^{(\ell_j)}$ consists of all 1's.
	\end{enumerate} Then Bowen's formula holds.
\end{theorem}

\begin{proof}
We endeavor to create a NCIFS $S$ by creating the alphabets at time $n$ given by
\begin{align*}
	E^{(1)}:=I^{\ell_1}\spand
	E^{(n)}:=I^{\ell_{n-1}+1,\ell_n} \quad\text{ for } n\geq 2.
\end{align*}
Since the incidence matrix $A^{(\ell_j)}$ consists of all 1's for each $j\in\NN$, this construction does in fact form a NCIFS. In fact, we claim that the first hypothesis implies that this construction forms a subexponentially bounded NCIFS. To see this we note that for $\ep>0$ and $n$ sufficiently large we have
\begin{align*}
	\#E^{(n)}&=\#I^{\ell_{n-1}+1,\ell_n}
	\leq \#I^{(\ell_{n-1}+1)}\cdots\#I^{(\ell_n)}\\
	&\leq e^{\ep(\ell_{n-1}+1)}\cdots e^{\ep\ell_n}
	=e^{\ep(\ell_{n-1}+1+\dots+\ell_n)}\\
	&=e^{\frac{\ep}{2}(\ell_n^2-\ell_{n-1}^2+\ell_n-\ell_{n-1})}.
\end{align*}
Thus taking a logarithm and dividing by $n$, we have
\begin{align*}
	\frac{1}{n}\log\#E^{(n)}&\leq\frac{\ep(\ell_n^2-\ell_{n-1}^2+\ell_n-\ell_{n-1})}{2n}
	=\frac{\ep(\ell_n^2-\ell_{n-1}^2)}{2n}+\frac{\ep(\ell_n-\ell_{n-1})}{2n}\\
	&\leq\frac{\ep(\ell_n^2-\ell_{n-1}^2)}{2n}+\frac{\ep(\ell_n^2-\ell_{n-1}^2)}{2n}
	\leq M\ep.
\end{align*}
Letting $\ep\to 0$ finishes the claim. Now we note that for each $n\in\NN$ we have
\begin{align*}	
	Z^{\Phi}_{\ell_n}(t)=Z^S_n(t).
\end{align*}
Thus we see that 
\begin{align*}
	\ul{P}^\Phi(t)\leq\ul{P}^S(t).
\end{align*}
Now we simply note that, since $J_\Phi=J_S$, we have that $\HD(J_\Phi)=\HD(J_S)$. Since $S$ is a subexponentially bounded system, Bowen's formula holds and so we have the inequality
\begin{align*}
	B_\Phi\leq B_S=\HD(J_S)=\HD(J_\Phi).
\end{align*}
As we already know that $B_\Phi\geq\HD(J_\Phi)$, we must in fact have that $B_\Phi=\HD(J_\Phi)$ as desired.
\end{proof}

\section{Bowen's Formula for Ascending NCGDMS}\label{Ascending}
In this section we introduce ascending NCGDMS, and show that Bowen's formula holds. Unlike in the previous section, we are able to remove both the subexponential boundedness and balancing conditions, however we still require that the system be finitely primitive. We also note that this class of systems extends all known results for finite and infinite NCIFS as we shall not require any restrictions on the alphabet size whatsoever, nor do we require that infinite systems be in class $\cM$. Furthermore, we will not even require that infinite systems be truly infinite, in that we allow for systems whose alphabets may be finite for some time and then infinite afterwards. 
\begin{definition}
	We say that a NCGDMS $\Phi$ is \emph{ascending} if the following hold.
	\begin{enumerate}
		\item $\cX^{(n)}\sub\cX^{(n+1)}$ for each $n\geq 0$.
		\item $I^{(n)}\sub I^{(n+1)}$ for each  $n\in\NN$.
		\item $\phi_i^{(n)}=\phi_{i}^{(n+1)}$ for each $i\in I^{(n)}$ and for each $n\in\NN$.
		\item $A^{(n)}_{ij}=A^{(n+1)}_{ij}$ for each $i,j\in I^{(n)}$ and for each $n\in\NN$.
	\end{enumerate}
\end{definition}
\begin{remark}
	Notice that if a $p$-finitely primitive NCGDMS is ascending then for any $n\in\NN$ and any $a\in I^{(n)}$ and $b\in I^{(n+j)}$ for any $j>p$ there is some word $\om$ such that $a\om b$ is admissible.
	
	In particular, since $I^{(1)}\sub I^{(j)}$ for every $j\in\NN$ we have that for each $a\in I^{(1)}$ and for each $n\geq p+1$ there is a word $\om$  with $\absval{\om}=n$ such that $a\om a$ is admissible. This property will allow us to construct an autonomous IFS.  
\end{remark}

\begin{remark}\label{rem: ascending GDMS implies ascending NIFS}
	Note that if $\Phi$ is an ascending NCGDMS such that the collections $\cX^{(n)}$ are a singleton for each $n\in\NN$ with $\cX^{(n)}=\cX^{(m)}$ for all $n,m\in\NN$ and each of the matrices $A^{(n)}$ consists only of ones, then we say that $\Phi$ is an ascending NCIFS. Unlike the graph directed setting, ascending NCIFS are, by definition, necessarily stationary. 
\end{remark}

\begin{theorem}\label{thm: BF for ascending NCGDMS}
	Suppose $\Phi$ is a finite ascending, $p$-finitely primitive NCGDMS. Then Bowen's formula holds.
\end{theorem} 
\begin{proof}
	Let $\ep> 0$ and $\ell\in\NN$ with $\ell>p$. Since the system is $p$-finitely primitive, for each $a\in I^{(1)}$, the set 
	\begin{align*}
		\Lambda_a^*=\set{\lm\in\Lambda_\ell:\om\lm a\in I^{\ell+p+1}, \om\in I^\ell, \om_1=a}
	\end{align*} 
	is nonempty. Then we are able to estimate the partition function so that we have 
	\begin{align*}
	Z_\ell(t)&=\sum_{\om\in I^\ell}\norm{D\phi_\om}^t
	=\sum_{a\in I^{(1)}}\sum_{\mathclap{\substack{\om\in I^{\ell}\\\om_1=a}}}\norm{D\phi_\om}^t
	\leq \#I^{(1)}\sum_{\mathclap{\substack{\om\in I^\ell\\\om_1=a_\ell}}}\norm{D\phi_\om}^t
	\end{align*}
	where $a_\ell\in I^{(1)}$, which, depending on $\ell$, maximizes the sum 
	\begin{align*}
	Z_\ell(t;(a,\spot)):=\sum_{a\in I^{(1)}}\sum_{\mathclap{\substack{\om\in I^{\ell}\\\om_1=a}}}\norm{D\phi_\om}^t.
	\end{align*}	
	Continuing our estimate we have 
	\begin{align*}
	Z_\ell(t)&\leq \frac{\#I^{(1)}}{\#\Lambda_{a_\ell}^*\cdot Q^t}\cdot \sum_{\mathclap{\substack{\om\in I^{\ell}\\\om_1=a_\ell}}}\norm{D\phi_\om}^t\cdot\sum_{\lm\in \Lambda^*_{a_\ell}}\norm{D\phi_{\lm}}^t
	= \frac{\#I^{(1)}}{\#\Lambda_{a_\ell}^*\cdot Q^t}\cdot \sum_{\mathclap{\substack{\om\in I^\ell\\ \om_1=a_\ell\\ \lm\in \Lambda_{a_\ell}^*}}} \norm{D\phi_{\om}}^t\cdot\norm{D\phi_{\lm}}^t\\
	&\leq\frac{\#I^{(1)}K^t}{\#\Lambda_{a_\ell}^*\cdot Q^t}\cdot \sum_{\mathclap{\substack{\om\in I^\ell\\\om_1=a_\ell\\ \lm\in \Lambda_{a_\ell}^*}}} \norm{D\phi_{\om\lm}}^t
	\leq \frac{\#I^{(1)}K^d}{Q^d}\cdot \sum_{\mathclap{\substack{\om\in I^\ell\\ \om_1=a_\ell\\ \lm\in \Lambda_{a_\ell}^*}}} \norm{D\phi_{\om\lm}}^t.
	\end{align*}
	As $\frac{\#I^{(1)}K^d}{Q^d}$ is some constant greater than 1 we can choose $\ell$, depending on $\ep$, sufficiently large such that the following hold.
	\begin{itemize}
		\item $Z_\ell(t)\geq e^{\ell(\ul{P}(t)-\frac{\ep}{2})}$.
		\item $\frac{\#I^{(1)}K^d}{Q^d}\leq e^{(\frac{\ell\ep}{2})}$.
	\end{itemize}
	In that case we have that
	\begin{align*}	
	\sum_{\mathclap{\substack{\om\in I^\ell\\ \om_1=a_\ell\\ \lm\in \Lambda^*}}} \norm{D\phi_{\om\lm}}^t&\geq \frac{Q^d}{\#I^{(1)}K^d} Z_\ell(t)\geq e^{\ell(\ul{P}(t)-\ep)}.
	\end{align*}
	Now for each $\ell\in\NN$, we define the alphabet 
	\begin{align*}
		E_\ell=\set{\gm\in I^{\ell+p}:\gm=\om\lm, \om\in I^\ell,\om_1=a_\ell,\lm\in \Lambda_{a_\ell}^*}.	
	\end{align*}
	Noting that $X_{i(a_\ell)}^{(0)}=X^{(0)}_{i(\gm)}=X_{t(\gm)}^{(\ell+p)}=X_{i(a_\ell)}^{(\ell+p)}$ for each $\gm\in E_\ell$, we see that 
	\begin{align*}
		S_\ell=\set{\phi_e: X_{i(a_\ell)}^{(0)}\to X_{i(a_\ell)}^{(0)}}_{e\in E_\ell}	
	\end{align*}
	forms an autonomous IFS for each $\ell>p$. Denote the partition and pressure functions associated with $S_\ell$ by $Z_m^{S_\ell}(t)$ and $P^{S_\ell}(t)$ respectively. 
	Then clearly we have
	\begin{align*}
	Z_m^{S_\ell}(t)&\leq Z_{m(\ell+p)}(t). 
	\end{align*}
	For $\gm\in E_\ell^m:=E_\ell\times\dots\times E_\ell$, the $m$-fold Cartesian product of $E_\ell$, we write $\gm=e_1e_2\dots e_m$, and thus we have  
	\begin{align*}
	Z_m^{S_\ell}(t)&=\sum_{\gm\in E_\ell^m}\norm{D\phi_\om}^t\\
	&\geq K^{-mt}\cdot\sum_{e_1,\dots,e_m\in E_\ell}\norm{D\phi_{e_1}}^t\cdots \norm{D\phi_{e_m}}^t\\
	&=K^{-mt}\cdot \sum_{e_1\in E_\ell}\norm{D\phi_{e_1}}^t\cdots \sum_{e_m\in E_\ell}\norm{D\phi_{e_m}}^t\\
	&=K^{-mt}\cdot\left(\sum_{e\in E_\ell}\norm{D\phi_e}^t\right)^m
	=K^{-mt}\cdot\left(\,\sum_{\mathclap{\substack{\om\in I^\ell\\ \om_1=a_\ell \\ \lm\in \Lambda_{a_\ell}^*}}}\norm{D\phi_{\om\lm}}^t\right)^m\\
	&\geq K^{-mt}\cdot\left(e^{\ell(\ul{P}(t)-\ep)}\right)^m
	=K^{-mt}\cdot e^{m\ell(\ul{P}(t)-\ep)}.
	\end{align*}
	Thus we have the combined inequality 
	\begin{align*}
	K^{-mt}\cdot e^{m\ell(\ul{P}(t)-\ep)}\leq Z^{S_\ell}_m(t)\leq Z_{m(\ell+p)}(t).
	\end{align*}
	As this inequality holds for all $m$ and all $\ell$ sufficiently large, taking logarithms and dividing by $m(\ell+p)$ gives
	\begin{align*}
	\frac{\ell}{\ell+p}\ul{P}(t)-\frac{\ell\ep}{\ell+p}-\frac{t\log K}{\ell+p}\leq \frac{1}{m(\ell+p)}\log Z^{S_\ell}_m(t)\leq \frac{1}{m(\ell+p)}\log Z_{m(\ell+p)}(t).
	\end{align*}
	Hence we must have
	\begin{align*}
		\frac{\ell}{\ell+p}\ul{P}(t)-\frac{\ell\ep}{\ell+p}-\frac{t\log K}{\ell+p}\leq \frac{1}{\ell+p}P^{S_\ell}(t)\leq \ul{P}(t).
	\end{align*}
	Noting that $\frac{\ell\ep}{\ell+p}<\ep$ we must also have
	\begin{align*}
		\frac{\ell}{\ell+p}\ul{P}(t)-\ep-\frac{t\log K}{\ell+p}\leq \frac{1}{\ell+p}P^{S_\ell}(t)\leq \ul{P}(t).
	\end{align*} 
	Since $S_\ell$ is an autonomous IFS, Bowen's formula holds, and so, $\HD(J_{S_\ell})=B_{S_\ell}$. Denote Bowen's parameter, $B_{S_\ell}$, for the system $S_\ell$ by $B_\ell$.
	
	Notice that letting $\ep\to 0$ we have that $\ell\to\infty$, thus for $t<B_\Phi$ and $\ep$ sufficiently small, we have $\frac{\ell}{\ell+p}\ul{P}(t)-\ep-\frac{t\log K}{\ell}>0$ and hence we have that $P^{S_\ell}(t)>0$ as well. Thus $t<B_\ell$, and as the right hand inequality gives that $B_\ell\leq B_\Phi$, we must have that $t<B_\ell\leq B_\Phi$. Letting $t\to B_\Phi$, and simultaneously $\ep\to 0$ and $\ell\to\infty$, we see that  
	\begin{align*}
		\limit{\ell}{\infty}B_\ell=B_\Phi.
	\end{align*} 
	Since $J_{S_\ell}\sub J_\Phi$ we have that 
	\begin{align*}
		\limit{\ell}{\infty}\HD(J_{S_\ell})\leq \HD(J_\Phi).
	\end{align*}
	Thus we have 
	\begin{align*}
		B_\Phi=\limit{\ell}{\infty}B_\ell=\limit{\ell}{\infty}\HD(J_{S_\ell})\leq \HD(J_\Phi),
	\end{align*}
	which, in light of the fact that $\HD(J_\Phi)\leq B_\Phi$, finishes the proof.	
\end{proof}

\begin{remark}
	It is worth pointing out that the technique used to create an autonomous IFS subsystem of the ascending NCGDMS can also be modified to create a NCIFS subsystem, $S$, of a subexponentially bounded and $p$-finitely primitive NCGDMS, $\Phi$, that is not necessarily ascending.  In this case we have that 
	\begin{align*}
	\HD(J_S)\leq\HD(J_\Phi).
	\end{align*}
	In fact, letting $\cA$$=\set{(a_n)_{n\in\NN}:a_j\in I^{(jp+1)}\text{ for }j=0,1,\dots}$ and letting $S_a$ be the NCIFS created along the sequence $a=(a_n)_{n\in\NN}$ we have that 
	\begin{align*}
	\sup_{a\in\cA}\HD(J_{S_a})\leq\HD(J_\Phi).
	\end{align*} 
	
\end{remark}

\begin{corollary}\label{cor: BF for ascending NIFS by NGDMS}
	In light of Remark \ref{rem: ascending GDMS implies ascending NIFS} we see that Bowen's formula holds for all finite ascending NCIFS as a corollary to Theorem \ref{thm: BF for ascending NCGDMS}. 
\end{corollary}

\begin{remark}	
	Notice that neither of the previous results mentioned the size nor the growth rates of the alphabets $I^{(n)}$. So in particular, these theorems apply for all finite NCGDMS and NCIFS which are not necessarily subexponentially bounded, extending the results of \cite{rempe-gillen_non-autonomous_2016}. Previously, Bowen's formula was only known for finite NCIFS which were subexponentially bounded and for infinite systems in class $\cM$. 
\end{remark}
In the proof of Theorem \ref{thm: BF for ascending NCGDMS} above, for each $\ell$ we were able to generate a finite autonomous iterated function system of size $\ell$. But we can generate an autonomous iterated function system from an ascending NCIFS in a different way. Consider the alphabet 
\begin{align*}
I_\infty:=\union_{n\geq 1} I^{(n)}.
\end{align*} 
Note that this alphabet may be either finite or infinite, depending on the number of unique functions in the system. Then $\Phi_\infty=\left(\phi_i\right)_{i\in I_\infty}$ is an autonomous conformal iterated function system generated by the ascending NCIFS $\Phi$.

Using this fact we have an alternate proof of Corollary \ref{cor: BF for ascending NIFS by NGDMS} for finite NCIFS. However, by modifying the main idea of the proof of the previous theorem in view of the previous comment about generating an infinite autonomous IFS from a NCIFS, we are able to get an extension to infinite NCIFS which are not necessarily of class \cM.
\begin{corollary}\label{cor: HD of ascending = to infinite generated}
	Suppose $\Phi$ is an ascending NCIFS, finite or infinite, which is not necessarily subexponentially bounded nor an element of class \cM. If $S_\infty$ is the autonomous IFS generated by $\Phi$ then 
	\begin{align*}
	\HD(J_\Phi)=\HD(J_{S_\infty}).
	\end{align*}
\end{corollary}

\begin{proof}
	
	Since $J_\Phi\sub J_{S_\infty}$ we have that $\HD(J_\Phi)\leq \HD(J_{S_\infty})$. Now to see the opposite inequality we recall that from Theorem 3.15 of \cite{mauldin_dimensions_1996} we have 
	\begin{align*}
	\HD(J_{S_\infty})=\sup\set{\HD(J_{S_F}):F\sub I_\infty, F \text{ finite}}
	\end{align*} 
	where $S_F$ is the finite autonomous IFS subsystem of $S_\infty$ considering only some finite alphabet $F\sub I_\infty$. If $I_\infty$ is finite then the above $\sup$ is actually the $\max$.
	
	But for each $F\sub I_\infty$ we can find $N_F\in\NN$ large enough such that $F\sub\cup_{1\leq j\leq n} I^{(j)}$ for each $n\geq N_F$. Thus taking $\ell>N_F$ and constructing the finite autonomous IFS $S_\ell$ as in the proof of the previous theorem we have that $J_{S_F}\sub J_{S_\ell}$, which of course implies that $\HD(J_{S_F})\leq\HD(J_{S_\ell})$. 
	
	Since we can find $\ell$ sufficiently large such that this holds for any finite $F\sub I_\infty$ we then must have that 
	\begin{align*}
	\HD(J_{S_\infty})=\sup\set{\HD(J_{S_F}):F\sub I_\infty, F \text{ finite}} \leq \limsup_{\ell\to\infty}\HD(J_{S_\ell})=\HD(J_\Phi),
	\end{align*}
	which completes the proof.
\end{proof}

\section{Examples and an Application to Conformal Dynamics}\label{Sec: Examples}
While our construction may seem a bit abstract, we assert that natural examples are abound. One such example comes from the theory of continued fractions. 
\begin{example}
	For each $n\in\NN$ let $I^{(n)}\sub\NN$ and consider the maps 
	\begin{align*}
		\phi_{b}^{(n)}:[0,1]\to [0,1]
	\end{align*} 
	defined by 
	\begin{align*}
		\phi_b^{(n)}(x)=\frac{1}{b+x}
	\end{align*}
	for each $b\in I^{(n)}$ and $n\in\NN$. Then these maps form a stationary NCIFS. By restricting which entries at time $n$ are allowed to follow which entries at time $n-1$, the system become a stationary NCGDMS.  
	
	For more on the uses of autonomous iterated function systems and graph directed Markov systems in the theory of continued fraction expansions see, for example, \cite{mauldin_dimensions_1996}, \cite{mauldin_conformal_1999},  \cite{mauldin_graph_2003}, etc.
\end{example}

\begin{example}
	As stated, Corollary \ref{cor: HD of ascending = to infinite generated} provides a method for finding the dimension of an ascending NCIFS, but from an alternate perspective, Corollary \ref{cor: HD of ascending = to infinite generated} shows that one can always create a whole family of examples of NCIFS for which the dimension is already known. Indeed, given any infinite autonomous IFS, one can select alphabets $I^{(n)}$ to form a stationary NCIFS. Furthermore, if they are chosen in such a way as to make the non-autonomous system ascending, then the dimension of the autonomous limit set is the same as for the non-autonomous system.  	
\end{example}

We now provide an application to conformal dynamics. In particular we provide an application of NCGDMS in their simplest form, stationary NCIFS, in order to find a general lower bound for the Hausdorff dimension of the Julia set of a non-autonomous dynamical system composed of elliptic functions. The following is based on the work of Kotus and Urba\'nski in \cite{kotus_hausdorff_2003}. A similar treatment for the random setting appears in \cite{roy_random_2011}.

Let $f_0:\CC\to\ol{\CC}$ be an elliptic function, i.e. a doubly periodic meromorphic function, which precisely means that there are two periods $\rho_1,\rho_2\in\CC$ such that for all $z\in\CC$ and $s,t\in\ZZ$ we have 
\begin{align*}
f_0(z)=f_0(z+s\rho_1+t\rho_2).
\end{align*}
As such, the periodic points of $f_0$ form a lattice in $\CC$ and thus there is a parallelogram $\cR$, which tessellates the plane, for which $f_0(\cR)=\hat{\CC}$. This parallelogram is then called a \textit{fundamental parallelogram} of $f_0$. Of course $\cR$ is not unique, and there are in fact infinitely many fundamental parallelograms of $f_0$. One such parallelogram is given by
\begin{align*}
\cR=\set{t_1\rho_1+t_2\rho_2:0\leq t_1,t_2\leq 1}.
\end{align*}
Now for each pole $b$ of $f_0$ we let $q_b$ denote its multiplicity and define 
\begin{align*}
q:=\max\set{q_b:b\in f_0^{-1}(\infty)}=\max\set{q_b:b\in f_0^{-1}(\infty)\intersect\cR}.
\end{align*}
We let $\Crit(f_0)$ denote the set of critical points of $f_0$. Now since $f_0$ is an elliptic function, $f_0$ has only finitely many critical points, and as such there is $R_0>0$ such that $\Crit(f_0)\sub B(0,R_0)$.
Now for given sequences $\seq{\lm_n},\seq{c_n}\sub\CC$, we let $f_n$ be the affine perturbation of $f_0$ defined by 
\begin{align*}
	f_n(z)=\lm_n\cdot f_0(z)+c_n.
\end{align*}
Define the non-autonomous map $F:=F_{\lm,c}:\CC\to\ol{\CC}$ whose iterates are given by 
\begin{align*}
	F^n_{\lm,c}(z)=f_n\circ\dots\circ f_1(z)
\end{align*}  
for each $n\in\NN$. 
We define the non-autonomous Fatou and Julia sets, respectively, in the usual way by taking $\cF(F)$ to be the set of points in $\CC$ for which the iterates $(F^n)$ form a normal family on some neighborhood, and $\jl(F)$ to be the complement of $\cF(F)$. 
Note that the functions $f_n$ share the same collections of periodic points, poles, and critical points as the  function $f_0$.  
\begin{theorem}
	Let $f_0$ be an elliptic function. Then there are $\ep,\dl>0$ such that if $\lm_n, \lm_n^{-1}\in B(1,\dl)$ and $c_n\in B(0,\ep)$ for all $n\in\NN$ then 
	\begin{align*}
		\HD(\jl(F))\geq \frac{2q}{q+1}.
	\end{align*}
\end{theorem}
\begin{proof}

	Let $B_R=\set{z\in\ol{\CC}:|z|>R}$, and for each pole $b$ of $f_0$ let $B_b(R)$ denote the connected component of $f_0^{-1}(B_R)$ which contains $b$. Recall that there exists $R_0>0$ such that $\Crit(f_0)\sub B(0,R_0)$. Now if $R>R_0$, then $B_R$ contains no critical values of $f_0$ which implies that the sets $B_b(R)$ are simply connected and mutually disjoint. For each pole $b$, there is a holomorphic function $G_b:B_b(R)\to \CC$ taking values in a neighborhood of 0 such that 
	\begin{align*}
	f_0(z)=\frac{G_b(z)}{(z-b)^{q_b}}.
	\end{align*}
	This fact immediately implies that there is some constant $L>0$ such that 
	\begin{align}\label{diam less than LR}
		\diam(B_b(R))\leq LR^{-\frac{1}{q_b}}\leq LR^{-\frac{1}{q}}
	\end{align}
	for each pole $b$. Choose $S<R_0$, and let $R_1>R_0$ sufficiently large such that 
	\begin{align}\label{f_0 diam leq S}
		LR^{-\frac{1}{q}}<S
	\end{align}
	for all $R\geq R_1$. For $U\sub B_R\bs\set{\infty}$ open and simply connected, the holomorphic inverse branches
	\begin{align*}
	f_{0,b,U,1}^{-1},\dots,f_{0,b,U,q_b}^{-1}
	\end{align*}  
	of $f_0$ are all well-defined on $U$ for each $1\leq j\leq q_b$. For all $z\in U$ we have 
	\begin{align*}
	\absval{(f_{0,b,U,j}^{-1})'(z)}\comp\absval{z}^{-\frac{q_b+1}{q_b}}.
	\end{align*}
	In light of \eqref{diam less than LR} and \eqref{f_0 diam leq S}, we see that for any two poles $b_1,b_2\in B_{2R_2}$ we have 
	\begin{align*}
		f_{0,b_2,b_1,j}^{-1}(B(b_1,R_0))\sub B_{b_2}(2R_2-R_0)\sub B_{b_2}(R_2)\sub B(b_2,S)\sub B(b_2,R_0).
	\end{align*}
	Consider the series 
	\begin{align*}
		\sum_{b\in B_{2R_2}\cap f_0^{-1}(\infty)}\absval{b}^{-r}.
	\end{align*}
	This series converges if and only if $r>2$. Inserting $r=t\frac{q+1}{q}$, we see that the series
	\begin{align*}
		\sum_{b\in B_{2R_2}\cap f_0^{-1}(\infty)}\absval{b}^{-t\frac{q+1}{q}}
	\end{align*}
	converges if and only if $t>2q/(q+1)$. Thus, for each $t\leq 2q/(q+1)$ there is some $N_t\in\NN$ such that 
	\begin{align}\label{sum over I'}
		\sum_{b\in I'}\absval{b}^{-t\frac{q+1}{q}}\geq 2K^2,
	\end{align} 
	where $I'\sub B_{2R_2}\intersect f_0^{-1}(\infty)$ finite with $\#I'=N_t$. 
	
	Fix a pole $a\in B_{2R_2}$ with $q_a=q$ and a finite subset $I\sub B_{2R_2}\intersect f_0^{-1}(\infty)$ with $q_b=q$ for each $b\in I$ and $\#I=N_t$ such that \eqref{sum over I'} holds summing over $I$. Now choose $\ep,\dl>0$ such that the following hold for all $b\in I$.
	\begin{itemize}
		\item $\dl<\frac{R_0-2S}{2S}$.
		\item $(1+\dl)\left(\dl(1+\absval{b})+ S\right)<\frac{R_0}{2}$.
		\item $\ep<\dl$.
	\end{itemize}
	To see that such a $\dl>0$ exists, we note that $\dl<\frac{R_0-2S}{2S}$ implies that $(1+\dl)S<R_0/2$. So, in particular, we have 
	\begin{align*}
		(1+\dl)\left(\dl(1+\absval{b})+ S\right)<\dl^2(1+\absval{b})+\dl(1+\absval{b})+\frac{R_0}{2}<R_0.
	\end{align*}
	Thus we see that we simply need to solve the quadratic inequality 
	\begin{align*}
		\dl^2+\dl<\frac{R_0}{2(1+M)}, 
	\end{align*} 
	where $M:=\max_{b\in I}\absval{b}$. However, we see that 
	\begin{align*}
		0<\dl<\frac{-1+\sqrt{1+2R_0(1+M)^{-1}}}{2}
	\end{align*}
	satisfies such an inequality. Now let $\lm:=\seq{\lm_n}$ and $c:=\seq{c_n}$ be sequences in $\CC$ such that $0\neq\lm_n,\lm_n^{-1}\in B(1,\dl)$ and $\absval{c_n}\leq \ep<\dl$ for each $n\in\NN$. Furthermore, for each $n\in\NN$ we define the functions
	\begin{align*}
		f_n:=\lm_n\cdot f_0+c_n.
	\end{align*}  
	In particular, by our choice of $\dl$ we have that for each $b\in I$ and $w\in B(b,S)$
	\begin{align}\label{S ball contained in R_0 ball}
		\frac{w-c_n}{\lm_n}\in B(b, R_0)	
	\end{align}
	for each $n\in \NN$. Indeed, we have that 
	\begin{align*}
		\absval{\frac{w-c_n}{\lm_n}-b}&\leq\absval{\lm_n^{-1}}\left(\absval{c_n}+\absval{w-b}+\absval{1-\lm_n}\absval{b}\right)\\
		&\leq(1+\dl)\left(\ep+S+\dl\absval{b}\right)\\
		&\leq(1+\dl)\left(\dl(1+\absval{b})+S\right),
	\end{align*}
	which is less than $R_0$ by our choice of $\ep$ and $\dl$. Now for each $b\in I$, fix inverse branches 
	\begin{align*}
	f_{0,b,a,1}^{-1}:\ol{B}(a,R_0)\to \ol{B}(b,S) \spand f_{0,a,b,1}^{-1}:\ol{B}(b,R_0)\to \ol{B}(a,S)
	\end{align*} 
	of $f_0$. 
	In view of \eqref{S ball contained in R_0 ball}, we see that for each $n\in\NN$ the inverse branches
	\begin{align*}
		f_{n,b,a,1}^{-1}:\ol{B}(a,S)\to \ol{B}(b,S) \spand f_{n,a,b,1}^{-1}:\ol{B}(b,S)\to \ol{B}(a,S)
	\end{align*}
	are well defined and given by 
	\begin{align*}
		f_{n,\spot,\spot,1}^{-1}(w)=f_{0,\spot,\spot,1}^{-1}\left(\frac{w-c_n}{\lm_n}\right).
	\end{align*}
	Now let 
	\begin{align*}
		\phi_b^{(n)}:=f_{2n-1,a,b,1}^{-1}\circ f_{2n,b,a,1}^{-1}:\ol{B}(a,S)\to\ol{B}(a,S)	
	\end{align*}
	for each $b\in I$. Setting $\Phi^{(n)}=\set{\phi_b^{(n)}:b\in I}$, then $\Phi=\set{\Phi^{(n)}}_{n\in\NN}$ forms a stationary NCIFS. 		
	Let $J_\Phi$ be the limit set of the NCIFS $\Phi$ given by
	\begin{align*}
	J_\Phi=	\intersect_{n\geq 0}\union_{\om\in I^n}\phi_\om\left(\ol{B}(a, S)\right).
	\end{align*}
	Then by construction we have that $J_\Phi\sub\jl(F)$ as for each $z\in J_\Phi$ and $\om\in I^\infty$
	\begin{align*}
	\absval{D\phi_{\om\rvert_n}^n(z)}\to 0 \quad \text{ as } n\to\infty.
	\end{align*}
	As $\Phi$ is uniformly finite, we have that Bowen's formula holds, and thus that 
	\begin{align*}
		\HD(J_\Phi)\leq\HD(\jl(F)).
	\end{align*}
	Thus in order to estimate a lower bound of the value on the right hand side it suffices to estimate a lower bound for $\HD(J_\Phi)$. In view of \eqref{thetaN lower bound ineq} it suffices to estimate $\ta_\Phi$. Recall that 
	\begin{align*}
	Z_{(n)}(t)=\sum_{b\in I}\norm{D\phi_b^{(n)}}^t.
	\end{align*} 
	Now we notice that 
	\begin{align*}
	Z_{(n)}(t)
	\comp \sum_{b\in I}\absval{a}^{-\frac{q+1}{q}t}\absval{b}^{-\frac{q+1}{q}t}
	\comp\sum_{b\in I}\absval{b}^{-\frac{q+1}{q}t}>2K^2.
	\end{align*}
	Now given that 
	\begin{align*}
		Z_n(t)\gtrsim K^{-nt}Z_{(1)}(t)\cdots Z_{(n)}(t)\gtrsim K^{n(2-t)}2^n\gtrsim 2^n,
	\end{align*}
	we see that $\ul{P}(t)>0$ and thus $\HD(J_\Phi)\geq t$. But as $t\leq \frac{2q}{q+1}$ was arbitrary, we have that 
	\begin{align*}
		\HD(J_\Phi)\geq \frac{2q}{q+1},
	\end{align*}
	which completes the proof.
\end{proof}
\begin{remark}
	In fact, the proof actually shows that the hyperbolic dimension of $F$ is at least $\frac{2q}{q+1}$. 
\end{remark}
In a forthcoming paper we present more general applications of NCGDMS, in their full generality to non-autonomous conformal dynamics, as well as results that generalize the previous theorem.

\bibliographystyle{jabbrv_plain}
\bibliography{mybib} 

\end{document}